\documentclass[pdftex,12pt]{amsart}

\usepackage{mathtools,amssymb,braket,booktabs}
\usepackage{hyperref}
\usepackage[capitalize]{cleveref}
\usepackage{autonum}

\newtheorem{theorem}{Theorem}[section]
\newtheorem{corollary}[theorem]{Corollary}
\newtheorem{proposition}[theorem]{Proposition}
\newtheorem{lemma}[theorem]{Lemma}
\newtheorem{conjecture}[theorem]{Conjecture}
\theoremstyle{definition}
\newtheorem{definition}[theorem]{Definition}
\newtheorem{algorithm}[theorem]{Algorithm}
\newtheorem{example}[theorem]{Example}

\newcommand{\F}{\mathbb{F}}
\newcommand{\Z}{\mathbb{Z}}
\newcommand{\Q}{\mathbb{Q}}
\newcommand{\C}{\mathbb{C}}
\newcommand{\B}{\mathbb{B}}
\newcommand{\Cl}[1]{\operatorname{Cl}_{#1}}
\newcommand{\gind}[2]{(#1:#2)}
\newcommand{\at}[1]{\widetilde{a}_{#1}}
\newcommand{\vt}[2]{\widetilde{v}_{#1}^{(#2)}}
\newcommand{\Vt}[2]{\widetilde{V}_{#1}^{(#2)}}
\newcommand{\Et}[1]{\widetilde{E}''_{#1}}
\newcommand{\jac}[2]{\genfrac{(}{)}{}{}{#1}{#2}}
\DeclareMathOperator{\Gal}{Gal}
\DeclareMathOperator{\sgn}{sgn}
\allowdisplaybreaks

\begin{document}

\title{On $\mathbb{Z}_2$-extensions of real quadratic fields}
\author[S.~Sasaki]{Sosuke Sasaki}
\address{
  Research Institute for Science and Engineering,
  Waseda University,
  3-4-1 Okubo, Shinjuku-ku, Tokyo 169-8555, Japan
}
\email{sosuke\_s5@asagi.waseda.jp}
\begin{abstract}
  Let $k$ be a real quadratic field.
  Let $k_n$ be the $n$-th layer of
  the $\mathbb{Z}_2$-extension of $k$,
  and $2^{a_n}$ be the $2$-part of the class number of $k_n$.
  In this paper,
  we derive many laws
  concerning the sequence $(a_0,a_1,\dotsc)$,
  by iteratively extending the group of cyclotomic units.
\end{abstract}
\maketitle

\section{Introduction}

Let $k$ be a real quadratic field, and
consider the (cyclotomic) $\Z_2$-extension
\[
  k=k_0\subseteq k_1\subseteq k_2\subseteq\dots\subseteq k_\infty=\bigcup_{n=0}^\infty k_n.
\]
The field $k_n$ is called the $n$-th layer of this $\Z_2$-extension,
which is the unique intermediate field of $k_\infty/k$ with degree $2^n$ over $k$.
Let $A_n$ be the $2$-Sylow subgroup of the ideal class group of $k_n$,
and let $2^{a_n}$ be the order of $A_n$.
From Iwasawa's class number formula~\cite{iwasawa_bull_1959,iwasawa_ann_1973} and
the Ferrero--Washington theorem~\cite{ferrerowashington_ann_1979},
there exist a non-negative integer $\lambda_2(k)$ and an integer $\nu_2(k)$
such that for any sufficiently large integer $n$,
\[
  a_n=\lambda_2(k)n+\nu_2(k)
\]
holds.
Under Greenberg's conjecture~\cite{greenberg_ajm_1976} we have $\lambda_2(k)=0$,
i.e., $a_n$ is eventually constant.
However, even in the case of $\Z_2$-extensions of real quadratic fields,
Greenberg's conjecture remains unsolved.
In \cite{fukudakomatsukumakawasasaki_jxm_2026},
it is shown that $\lambda_2(\Q(\sqrt m))=0$ for all
$1<m<10^6$.

Our aim in the present paper is to understand
the behavior of $a_n$ not only for sufficiently large $n$
but also for small $n$.

What kinds of integer sequences could there be for $(a_0,a_1,\dotsc)$?
Of course, not just any sequence will do;
there are many restrictions.
Iwasawa's class number formula and the Ferrero--Washington theorem
are among such restrictions.
Here are other examples of well-known restrictions.
Let $k=\Q(\sqrt m)$, where $m>1$ is a square-free odd integer.
The class field theory yields that
the sequence $(a_0,a_1,\dotsc)$ is non-decreasing.
Further, Fukuda's theorem~\cite{fukuda_japan_1994}
states that if $a_n=a_{n+1}$, then $a_i=a_n$ for all $i\ge n$.
We wish to find many restrictions on the sequence $(a_0,a_1,\dotsc)$
under certain conditions.
In this paper, we give a general method for this purpose.
For example, we obtained the following restriction
(see \cref{ex:table}).
\begin{theorem}
  \label{thm:kumakawa}
  Let $k=\Q(\sqrt p)$, where $p$ is a prime with $p\equiv17\pmod{32}$.
  If $a_1=1$ and $a_2=3$, then $a_3=3$
  (and hence $a_i=3$ for all $i\ge2$ from Fukuda's theorem).
\end{theorem}
This theorem had been partially discovered by Kumakawa
in an informal seminar.
% Similar to this theorem,
% we wish to find many restrictions on the sequence $(a_0,a_1,\dotsc)$
% under certain conditions.
% In this paper, we give a general method for this purpose.
As other examples,
we obtained the following results
(see \cref{thm:a1a2_n,thm:a1a2_p}).
\begin{theorem}
  Let $k=\Q(\sqrt p)$, where $p$ is a prime with $p\equiv1\pmod{16}$.
  Assume that the norm of the fundamental unit of $\Q(\sqrt{2p})$ is $-1$.
  % Then we have $a_1\ge2$.
  % Moreover, if $a_1\ge3$ and $N_{\Q(\zeta_{8p})/k(\zeta_8)}(1-\zeta_{8p})<0$,
  If $a_1\ge3$ and $N_{\Q(\zeta_{8p})/k(\zeta_8)}(1-\zeta_{8p})<0$,
  then $a_2=a_1$ (and hence $a_i=a_1$ for all $i\ge1$ from Fukuda's theorem,
  which yields $\lambda_2(k)=0$).
\end{theorem}
\begin{theorem}
  Let $k=\Q(\sqrt p)$, where $p$ is a prime with $p\equiv1\pmod{16}$.
  Assume that the norm of the fundamental unit of $\Q(\sqrt{2p})$ is $1$.
  If $a_1\ge1$, then $a_2>a_1$.
  Moreover, if $a_1\ge2$ and $N_{\Q(\zeta_{8p})/k(\zeta_8)}(1-\zeta_{8p})<0$,
  then $a_2=a_1+1$ and $\lambda_2(k)\le1$.
\end{theorem}

In \cref{sect:sinnott},
we summarize the fundamental properties of
the specific units, $c_n$ and $\eta_n$.
In \cref{sect:iter},
we define the iteration,
which forms the central framework of this paper.
In \cref{sect:fund},
we introduce some patterns observed in the iteration.
Although not strictly necessary for our purposes,
we will obtain several interesting general properties of the iteration.
In \cref{sect:sq},
we explain how to determine whether a unit is square,
which is necessary during the iteration.
\cref{thm:main} is a key theorem,
which we consider a very interesting result.
We also outline a method for speeding up the algorithm
in \cite{fukudakomatsukumakawasasaki_jxm_2026} using this theorem.
In \cref{sect:prime},
we describe a method for narrowing down the candidates
for the sequence $(a_0,a_1,\dotsc)$
when $m$ is a prime, without $m$ being explicitly given.
As a result, we will obtain the theorems stated above.

In the following, we consider any number fields as subfields of $\C$.
We put $\zeta_t=\exp(2\pi\sqrt{-1}/t)$ for a positive integer $t$.
For a number field $F$, we write
$\Cl{F}$, $h_F$, and $E_F$ for the ideal class group, the class number, and the unit group of $F$, respectively.
Let $\B_n$ be the maximal real subfield of $\Q(\zeta_{2^{n+2}})$ for $n\ge0$,
i.e., $\B_n=\Q(\zeta_{2^{n+2}}+\zeta_{2^{n+2}}^{-1})$.
Then $\B_\infty=\bigcup_{n\ge0}\B_n$ is the cyclotomic
$\Z_2$-extension of $\Q$,
and each $\B_n$ is the $n$-th layer of the extension.

Let $k=\Q(\sqrt m)$ be a real quadratic field
with a square-free integer $m>1$,
and put $k_n=k\B_n$.
Then $k_\infty=k\B_\infty=\bigcup_{n\ge0}k_n$ is the cyclotomic
$\Z_2$-extension of $k$.
Let $\gamma$ be the topological generator of the Galois group $\Gal(k_\infty/k)$
which satisfies $(\zeta_{2^{n+2}}+\zeta_{2^{n+2}}^{-1})^\gamma=\zeta_{2^{n+2}}^3+\zeta_{2^{n+2}}^{-3}$ for all $n$
\footnote{
  Usually $\gamma$ is taken so that $\zeta_{2^{n+2}}+\zeta_{2^{n+2}}^{-1}\mapsto\zeta_{2^{n+2}}^5+\zeta_{2^{n+2}}^{-5}$,
  but here it is taken this way to ensure compatibility with the definition of $c_n$
  (see the proof of \cref{prop:norm_tau}).
},
and let $\tau$ be the generator of $\Gal(k_\infty/\B_\infty)$.

We write $E_n=E_{k_n}$ for convenience.
Also we write the norm maps as $N_{n',n}=N_{k_{n'}/k_n}\colon E_{n'}\to E_n$
for $n\le n'$.

For $n\ge0$, let $A_n$ be the Sylow $2$-subgroup of the class group $\Cl{k_n}$,
and let $a_n$ be the non-negative integer satisfying $\lvert A_n\rvert=2^{a_n}$.
Our main purpose is to find some properties of the sequence $(a_n)_{n\ge0}$.

We will omit the case where $m$ is even,
since replacing $m$ with $m/2$ has no effect other than on $a_0$.
Also we can ignore the following trivial cases
(see \cite{iwasawa_abh_1956,ozakitaya_manu_1997}).
\begin{theorem}
  If $m$ is a prime with $m\not\equiv1\pmod8$,
  then $a_n=0$ for all $n\ge0$.
\end{theorem}

In the following,
we fix $m>1$
assuming $m$ is odd, square-free, and
not a prime with $m\not\equiv1\pmod8$.

\section{Cyclotomic Units}
\label{sect:sinnott}

In this section, we shall define cyclotomic units
denoted by $c_n$ and $\eta_n$,
and summarize their fundamental properties.
In the case $m$ is a prime,
their properties are summarized in \cite{fukudakomatsu_funct_2014},
hence we extend them to the general $m$.
First, we define the following units.
\begin{definition}
  Let $n\ge0$.
  \begin{enumerate}
    \item Let $c_n=1+\zeta_{2^{n+2}}+\zeta_{2^{n+2}}^{-1}$.
    \item Let
    \[
      \eta_n=\zeta_{2^{n+2}}^{-\varphi(m)s/4}N_{\Q(\zeta_{2^{n+2}m})/k(\zeta_{2^{n+2}})}(1-\zeta_{2^{n+2}m}),
    \]
    where $s$ is an integer with $ms\equiv1\pmod{2^{n+2}}$,
    and $\varphi$ is Euler's totient function.
  \end{enumerate}
\end{definition}
Note that from the assumption on $m$,
we see that $\varphi(m)$ is divisible by $4$.

It is well-known that $c_n$ is a unit of $\B_n$.
The element $\eta_n$ belongs to Sinnott's circular unit group of $k(\zeta_{2^{n+2}})=k_n(\sqrt{-1})$,
but is actually a real number (see \cite[Lemma~2.1]{fukudakomatsukumakawasasaki_jxm_2026}).
\begin{proposition}
  \label{prop:real}
  For any $n\ge0$, $\eta_n$ is a unit of $k_n$.
\end{proposition}

Next, we construct the following specific subgroups of $E_n$ using $c_n$ and $\eta_n$.
\begin{definition}
  We define the subgroups
  \begin{align}
    C_n&=\braket{c_n,c_n^\gamma,c_n^{\gamma^2},\dots,c_n^{\gamma^{2^{n-1}-1}}}\subseteq E_{\B_n},\\
    H_n&=\braket{\eta_n,\eta_n^\gamma,\eta_n^{\gamma^2},\dots,\eta_n^{\gamma^{2^{n-1}-1}}}\subseteq E_n
  \end{align}
  for $n\ge1$.
  Further, we put
  \[
    E'_n=E_0\times C_1\times H_1\times C_2\times H_2\times\dots\times C_n\times H_n\subseteq E_n
  \]
  for $n\ge0$.
\end{definition}

Using Sinnott's formula~\cite[Theorem~4.1]{sinnott_invent_1980},
we can deduce the following relationship between
the unit indices and the class numbers (see \cite[\S2, \S8]{fukudakomatsukumakawasasaki_jxm_2026}).
\begin{theorem}
  \label{thm:sinnott}
  For any $n\ge0$, we have
  \[
    \gind{E_n}{E'_n}=\frac{h_{k_n}}{h_k}.
  \]
\end{theorem}

We want to know the $2$-part of this.
\begin{definition}
  Let $n\ge0$.
  \begin{enumerate}
    \item Let $E''_n$ be the subgroup of $E_n$
    containing $E'_n$
    so that $E''_n/E'_n$ is the Sylow $2$-subgroup of $E_n/E'_n$,
    i.e., $\gind{E_n}{E''_n}$ is odd and $\gind{E''_n}{E'_n}$ is $2$-power.
    \item Let $V_0=E_0$ and $V_n=E''_{n-1}\times C_n\times H_n$ for $n\ge1$.
    We have $E'_n\subseteq V_n\subseteq E''_n$.
  \end{enumerate}
\end{definition}

\begin{corollary}
  \label{cor:sinnott}
  \(
    \gind{E''_n}{V_n}=2^{a_n-a_{n-1}}
  \) for $n\ge1$.
\end{corollary}
\begin{proof}
  We obtain
  $\gind{E''_n}{E'_n}=2^{a_n-a_0}$ and
  $\gind{V_n}{E'_n}=\gind{E''_{n-1}}{E'_{n-1}}=2^{a_{n-1}-a_0}$
  from \cref{thm:sinnott}.
\end{proof}

Next we state the following formulae for the norms of $c_n$ and $\eta_n$.
The norms to lower layers are simple.
\begin{proposition}
  \label{prop:norm_gam}
  For $n\ge1$, we have the following formulae.
  \begin{enumerate}
    \item $N_{n,n-1}(c_n)=-c_{n-1}$.
    \item $N_{n,n-1}(\eta_n)=(-1)^{\varphi(m)/4}\eta_{n-1}$.
  \end{enumerate}
\end{proposition}
\begin{proof}
  By straightforward calculations.
\end{proof}

The norms to $\B_n$ are somewhat difficult,
but we have the following explicit formula.
\begin{proposition}
  \label{prop:norm_tau}
  Let $n\ge0$.
  For a divisor $d$ of $m$, let $r_d$ be
  a non-negative integer satisfying $d\equiv\pm3^{r_d}\pmod{2^{n+3}}$.
  Then we have
  \[
    \eta_n^{1+\tau}=\prod_{d\mid m}c_n^{-\mu(d)(\gamma^{-1}+\gamma^{-2}+\dots+\gamma^{-r_d})},
  \]
  where $\mu$ is the M\"obius function.
\end{proposition}
\begin{proof}
  Let $s,t$ be integers with $ms+2^{n+3}t=1$.
  We have
  \begin{align}
    \eta_n^{1+\tau}&=N_{k_n/\B_n}\Bigl(\zeta_{2^{n+2}}^{-\varphi(m)s/4}N_{\Q(\zeta_{2^{n+2}m})/k(\zeta_{2^{n+2}})}(1-\zeta_{2^{n+2}m})\Bigr)\\
    &=\zeta_{2^{n+2}}^{-\varphi(m)s/2}N_{\Q(\zeta_{2^{n+2}m})/\Q(\zeta_{2^{n+2}})}(1-\zeta_{2^{n+2}}^s\zeta_m^t)\\
    &=\zeta_{2^{n+2}}^{\varphi(m)s/2}N_{\Q(\zeta_{2^{n+2}m})/\Q(\zeta_{2^{n+2}})}(\zeta_{2^{n+2}}^{-s}-\zeta_m^t)\\
    &=\zeta_{2^{n+2}}^{\varphi(m)s/2}\Phi_m(\zeta_{2^{n+2}}^{-s}),
  \end{align}
  where $\Phi_m$ is the $m$-th cyclotomic polynomial.
  From the M\"obius inversion formula, we know
  \[
    \Phi_m(X)=\prod_{d\mid m}(X^{m/d}-1)^{\mu(d)},
  \]
  and also its homogenization
  \[
    Y^{\varphi(m)}\Phi_m(X/Y)=\prod_{d\mid m}(X^{m/d}-Y^{m/d})^{\mu(d)}.
  \]
  Substituting $X=\zeta_{2^{n+3}}^{-s}$ and $Y=\zeta_{2^{n+3}}^s$
  into this yields
  \[
    \eta_n^{1+\tau}=\zeta_{2^{n+2}}^{\varphi(m)s/2}\Phi_m(\zeta_{2^{n+2}}^{-s})
    =\prod_{d\mid m}(\zeta_{2^{n+3}}^{-ms/d}-\zeta_{2^{n+3}}^{ms/d})^{\mu(d)}.
  \]
  From the assumption on $r_d$,
  we obtain $3^{r_d}\equiv(-2/d)d\pmod{2^{n+3}}$
  since $(-2/3)=1$,
  where $(\,\cdot\,/\,\cdot\,)$ is the Jacobi symbol.
  Hence from $ms\equiv1\pmod{2^{n+3}}$,
  \[
    \zeta_{2^{n+3}}^{-ms/d}-\zeta_{2^{n+3}}^{ms/d}=\jac{-2}{d}(\zeta_{2^{n+3}}^{-1}-\zeta_{2^{n+3}})^{\gamma^{-r_d}},
  \]
  where the domain of $\gamma$ is extended as $\zeta_{2^{n+3}}^\gamma=\zeta_{2^{n+3}}^3$.
  On the other hand,
  \[
    (\zeta_{2^{n+3}}^{-1}-\zeta_{2^{n+3}})^{\gamma-1}
    =\frac{\zeta_{2^{n+3}}^{-3}-\zeta_{2^{n+3}}^3}{\zeta_{2^{n+3}}^{-1}-\zeta_{2^{n+3}}}
    =1+\zeta_{2^{n+3}}^2+\zeta_{2^{n+3}}^{-2}=c_n.
  \]
  Therefore we obtain
  \[
    c_n^{-(\gamma^{-1}+\gamma^{-2}+\dots+\gamma^{-r_d})}
    =(\zeta_{2^{n+3}}^{-1}-\zeta_{2^{n+3}})^{\gamma^{-r_d}-1}
    =\jac{-2}{d}\frac{\zeta_{2^{n+3}}^{-ms/d}-\zeta_{2^{n+3}}^{ms/d}}{\zeta_{2^{n+3}}^{-1}-\zeta_{2^{n+3}}},
  \]
  and thus
  \begin{align}
    \prod_{d\mid m}c_n^{-\mu(d)(\gamma^{-1}+\gamma^{-2}+\dots+\gamma^{-r_d})}
    &=\prod_{d\mid m}\jac{-2}{d}\left( \frac{\zeta_{2^{n+3}}^{-ms/d}-\zeta_{2^{n+3}}^{ms/d}}{\zeta_{2^{n+3}}^{-1}-\zeta_{2^{n+3}}} \right)^{\mu(d)}\\
    &=\prod_{d\mid m}\jac{-2}{d}(\zeta_{2^{n+3}}^{-ms/d}-\zeta_{2^{n+3}}^{ms/d})^{\mu(d)}
  \end{align}
  since $\sum_{d\mid m}\mu(d)=0$.
  Hence, it suffices to show that $\prod_{d\mid m}(-2/d)=1$.
  If $m$ is prime, we have $m\equiv1\ (8)$ from our assumption,
  which yields $(-2/m)=1$.
  If $m$ is composite,
  \[
    \prod_{d\mid m}\jac{-2}{d}
    =\prod_{d\mid m}\prod_{p\mid d}\jac{-2}{p}
    =\prod_{p\mid m}\jac{-2}{p}^{\sigma_0(m/p)}=1,
  \]
  where $p$ runs through the prime divisors
  and $\sigma_0(m/p)$ is the number of the divisors of $m/p$
  (equivalently, the number of $d$ with $p\mid d\mid m$),
  which is even.
\end{proof}

\begin{example}
  \label{ex:norm_tau}
  Assume that $m=p$ is a prime (with $p\equiv1\pmod8$).
  Then
  \[
    \eta_n^{1+\tau}=c_n^{\gamma^{-1}+\gamma^{-2}+\dots+\gamma^{-r_p}}.
  \]
  Using \cref{prop:norm_gam} we obtain
  \begin{align}
    \eta_1^{1+\tau}&=\begin{cases}
      1&(p\equiv1\pmod{16})\\
      -1&(p\equiv9\pmod{16}),
    \end{cases}\\
    \eta_2^{1+\tau}&=\begin{cases}
      1&(p\equiv1\pmod{32})\\
      -c_2^{-1-\gamma}&(p\equiv9\pmod{32})\\
      -1&(p\equiv17\pmod{32})\\
      c_2^{-1-\gamma}&(p\equiv25\pmod{32}).
    \end{cases}
  \end{align}
\end{example}

The unit $\eta_0$ can be expressed explicitly as follows.
\begin{proposition}
  \label{prop:eta0}
  Let $\varepsilon>1$ be the fundamental unit of $k$.
  Then,
  \[
    \eta_0=\begin{cases}
      \pm1&(m\equiv1\pmod8)\\
      \pm\varepsilon^{-h_k}&(m\equiv5\pmod8)\\
      \pm\varepsilon^{-h_k/2}&(m\equiv3\pmod4).
    \end{cases}
  \]
\end{proposition}
\begin{proof}
  We use the following famous class number formula for real quadratic fields
  (see e.g.~\cite[Theorem~152]{hecke_gtm_1981}):
  \[
    \varepsilon^{2h_k}=\prod_{x=1}^{d-1}(\zeta_{2d}^x-\zeta_{2d}^{-x})^{-\chi(x)},
  \]
  where $d$ is the discriminant of $k$
  and $\chi$ is the quadratic character associated to $k$.
  Let
  \[
    \alpha=N_{\Q(\zeta_d)/k}(1-\zeta_d)=\prod_{\begin{smallmatrix}
      1\le x<d\\\chi(x)=1
    \end{smallmatrix}}(1-\zeta_d^x).
  \]
  From the above formula,
  $\varepsilon^{2h_k}$ equals to $\alpha^{\tau-1}$
  up to a root of unity.
  Since $1-\zeta_d$ is totally imaginary,
  $\alpha$ is totally positive.
  Hence we obtain $\varepsilon^{2h_k}=\alpha^{\tau-1}$.

  \cref{prop:real} yields
  \[
    \eta_0^2=N_{k(\zeta_4)/k}(\eta_0)=N_{\Q(\zeta_{4m})/k}(1-\zeta_{4m}).
  \]

  Assume $m\equiv3\ (4)$.
  Then $\eta_0^2=\alpha$.
  Since $\alpha^{1+\tau}=N_{k/\Q}(\eta_0)^2=1$,
  we obtain
  \[
    \varepsilon^{2h_k}=\alpha^{\tau-1}=\alpha^{-2}=\eta_0^{-4}.
  \]

  Assume $m\equiv1\ (4)$.
  Then
  \begin{align}
    \eta_0^2&=N_{\Q(\zeta_m)/k}(N_{\Q(\zeta_{4m})/\Q(\zeta_m)}(1-\zeta_{4m}))\\
    &=N_{\Q(\zeta_m)/k}((1-\zeta_{4m})(1+\zeta_{4m}))\\
    &=N_{\Q(\zeta_m)/k}(1-\zeta_{2m})\\
    &=N_{\Q(\zeta_m)/k}(1+\zeta_m^{\frac{1+m}{2}})\\
    &=\begin{cases}
      N_{\Q(\zeta_m)/k}(1+\zeta_m)&(\chi(2)=1),\\
      N_{\Q(\zeta_m)/k}(1+\zeta_m)^\tau&(\chi(2)=-1).
    \end{cases}
  \end{align}
  If $m\equiv1\ (8)$ we have
  \[
    \eta_0^2\alpha=N_{\Q(\zeta_m)/k}(1-\zeta_m^2)=\alpha,
  \]
  hence $\eta_0^2=1$.
  If $m\equiv5\ (8)$ we have
  \[
    \eta_0^2\alpha^\tau=N_{\Q(\zeta_m)/k}(1-\zeta_m^2)^\tau=\alpha,
  \]
  hence $\eta_0^2=\alpha^{1-\tau}=\varepsilon^{-2h_k}$.
\end{proof}

Similarly, the following formula can be derived.
This will be used in \cref{sect:prime}.
\begin{proposition}
  \label{prop:norm_eta1}
  \(
    N_{k_1/k'}(\eta_1)=\pm(\varepsilon')^{-h_{k'}/2}
  \),
  where $k'=\Q(\sqrt{2m})$
  and $\varepsilon'>1$ is the fundamental unit of $k'$.
\end{proposition}
\begin{proof}
  Similarly as the proof of \cref{prop:eta0},
  we can show $(\varepsilon')^{2h_{k'}}=\alpha^{\tau-1}$ and
  \[
    N_{k_1/k'}(\eta_1)^2=N_{k(\zeta_8)/k'}(\eta_1)=N_{\Q(\zeta_{8m})/k'}(1-\zeta_{8m})=\alpha.
  \]
  Since $\alpha^{1+\tau}=N_{k_1/\Q}(\eta_1)^2=1$,
  we have
  \[
    (\varepsilon')^{2h_{k'}}=\alpha^{\tau-1}=\alpha^{-2}=N_{k_1/k'}(\eta_1)^{-4}.
  \]
\end{proof}

Next we state a formula for the signs of the conjugates of $c_n$.
\begin{proposition}
  \label{prop:sign_c}
  Let $i$ be a non-negative integer,
  and $d$ be the integer with $0<d<2^{n+2}$ and $d\equiv3^i\pmod{2^{n+2}}$.
  Then we have
  \[
    \sgn c_n^{\gamma^i}=\begin{cases}
      -1&(2^{n+2}/3<d<2^{n+3}/3),\\
      1&(\text{otherwise}).
    \end{cases}
  \]
\end{proposition}
\begin{proof}
  Extend the domain of $\gamma$ as $\zeta_{2^{n+3}}^\gamma=\zeta_{2^{n+3}}^3$.
  Since $\zeta_{2^{n+2}}^{\gamma^i}=\zeta_{2^{n+2}}^d$,
  we have
  $\zeta_{2^{n+3}}^{\gamma^i}=\pm\zeta_{2^{n+3}}^d$.
  Hence
  \[
    c_n^{\gamma^i}
    =(1+\zeta_{2^{n+3}}^2+\zeta_{2^{n+3}}^{-2})^{\gamma^i}
    =\left( \frac{\zeta_{2^{n+3}}^{-3}-\zeta_{2^{n+3}}^3}{\zeta_{2^{n+3}}^{-1}-\zeta_{2^{n+3}}} \right)^{\gamma^i}
    =\frac{\zeta_{2^{n+3}}^{-3d}-\zeta_{2^{n+3}}^{3d}}{\zeta_{2^{n+3}}^{-d}-\zeta_{2^{n+3}}^d}
  \]
  The argument of the denominator is $-\pi/2$.
  It is easy to see that the argument of the numerator
  is $\pi/2$ if and only if $2^{n+2}/3<d<2^{n+3}/3$.
\end{proof}

% \begin{proposition}
%   \label{prop:sign_eta}
%   Let $i$ be a non-negative integer.
%   We define a subset of $\Z$ as
%   \[
%     Z_i=\set{0<z<m|((-1)^iz/m)=1,\,z\equiv mj\pmod{2^{n+2}}\text{ for some }1\le j\le s_i},
%   \]
%   where $s_i$ is the integer satisfying $0<s_i<2^{n+2}$ and $ms_i\equiv3^i\pmod{2^{n+2}}$.
%   Then we have
%   \[
%     \sgn\eta_n^{\gamma^i}=\begin{cases}
%       (-1)^{\#Z_i}&(m\equiv1\pmod4)\\
%       (-1)^{\#Z_i+h_{\Q(\sqrt{-m})}/2}&(m\equiv3\pmod4).
%     \end{cases}
%   \]
% \end{proposition}

In later sections, we shall consider whether
a unit is a perfect square.
For a unit of $\B_n$,
being square in $\B_n$ and being square in $k_n$ are equivalent.
\begin{proposition}
  \label{prop:index_tau}
  For $n\ge0$, we have
  $E_{\B_n}\cap E_n^2=E_{\B_n}^2$.
\end{proposition}
\begin{proof}
  Take a prime divisor $p$ of $m$ and
  a prime ideal $\mathfrak{p}$ of $\B_n$ above $p$.
  Let
  $u\in E_{\B_n}\cap E_n^2\smallsetminus E_{\B_n}^2$.
  Then $\B_n\subsetneq\B_n(\sqrt u)\subseteq k_n$,
  hence $\B_n(\sqrt u)=k_n$.
  From the Kummer theory, we have
  $\sqrt{mu}\in\B_n$.
  Hence the $\mathfrak{p}$-adic valuation of $m$ must be even,
  since $u$ is a unit.
  This contradicts that $m$ is square-free and $\B_n/\Q$ is unramified outside $2$.
\end{proof}

However, for a unit of $k_n$,
being square in $k_n$ and being square in higher layers are not equivalent in general.
\begin{proposition}
  \label{prop:index_gam}
  Let $n\ge0$.
  \begin{enumerate}
    \item We have $\gind{E_n\cap E_{n+1}^2}{E_n^2}=1\text{ or }2$,
    i.e., if there is $u\in E_n$ which is not square in $E_n$ but becomes square in $E_{n+1}$,
    then such $u$ is unique up to $E_n^2$.
    \label{it:index_12}
    \item If $m\equiv1\pmod4$, then $E_n\cap E_{n+1}^2=E_n^2$.
    \label{it:index_mod}
    \item If $n<n'$, then $\gind{E_n\cap E_{n'}^2}{E_n^2}=\gind{E_n\cap E_{n+1}^2}{E_n^2}$.
    \label{it:index_jump}
    \item If $n\le n'$, then $\gind{E_n\cap E_{n+1}^2}{E_n^2}\ge\gind{E_{n'}\cap E_{n'+1}^2}{E_{n'}^2}$.
    \label{it:index_dec}
    \item $\gind{E_n\cap E_{n+1}^2}{E_n^2}=\gind{E''_n\cap(E''_{n+1})^2}{(E''_n)^2}$.
    \label{it:index_pp}
  \end{enumerate}
\end{proposition}
\begin{proof}
  Let $n<n'$ and let $u\in E_n\cap E_{n'}^2\smallsetminus E_n^2$.
  Then $k_n(\sqrt u)$ is a quadratic extension of $k_n$
  contained in $k_{n'}$,
  which must be $k_{n+1}$.
  Hence $u\in E_{n+1}^2$, and
  such a unit $u$ is unique up to $E_n^2$ from the Kummer theory.
  Thus we have (\labelcref{it:index_12}) and (\labelcref{it:index_jump}).

  It is well-known that $k_{n+1}=k_n(\sqrt{\beta_n})$,
  where $\beta_n=2-\zeta_{2^{n+2}}-\zeta_{2^{n+2}}^{-1}$.
  Hence from the Kummer theory,
  we have $\sqrt{\beta_nu}\in k_n$.
  The ideal $(\beta_n)$
  is the unique prime ideal of $\B_n$ above $2$.
  Take a prime ideal $\mathfrak{p}$ of $k_n$ above $(\beta_n)$.
  Since $u$ is a unit and $\sqrt{\beta_nu}\in k_n$,
  the $\mathfrak{p}$-adic valuation of $\beta_n$ must be even.
  If $m\equiv1\pmod4$, this contradicts that $k_n/\B_n$ is unramified above $2$.
  This yields (\labelcref{it:index_mod}).

  Let $n<n'$ and let $u\in E_{n'}\cap E_{n'+1}^2\smallsetminus E_{n'}^2$.
  Then we have $\sqrt{\beta_{n'}u}\in k_{n'}$.
  Since $N_{n',n}(\beta_{n'})=\beta_n$,
  we have $\sqrt{\beta_nN_{n',n}(u)}\in k_n$,
  which yields (\labelcref{it:index_dec}).

  Finally, it is easy to show that
  \[
    (E''_n\cap(E''_{n+1})^2)\big/(E''_n)^2\simeq(E_n\cap E_{n+1}^2)\big/E_n^2
  \]
  via the natural map.
  Hence we have (\labelcref{it:index_pp}).
\end{proof}

\section{Iterative extensions of unit groups}
\label{sect:iter}

Our goal is to construct $E''_n$ by
iteratively extending $V_n$:
\[
  V_n=V_n^{(0)}\subseteq V_n^{(1)}\subseteq\dots\subseteq V_n^{(a_n-a_{n-1})}=E''_n,
\]
with $\gind{V_n^{(j)}}{V_n^{(j-1)}}=2$ for $1\le j\le a_n-a_{n-1}$.
In this iteration, we will also assume
that each $V_n^{(j)}$ is $\gamma$-invariant.
It is possible from the following lemma.
\begin{lemma}
  \label{lem:key_gam}
  Let $V$ be a subgroup of $E_n$ containing $-1$
  and satisfying $V^\gamma=V$.
  Assume that the index $\gind{E_n}{V}$ is even.
  Then there exists
  $v\in V\cap E_n^2$ satisfying $v^{1+\gamma}\in V^2$ and $v\notin V^2$.
  For this $v$, we have $\braket{V,\sqrt v}^\gamma=\braket{V,\sqrt v}$.
\end{lemma}
\begin{proof}
  Take $u\in E_n\smallsetminus V$ with $u^2\in V$.
  If $u^2\in V^2$, then we have $u\in V$ from $-1\in V$,
  which is absurd.
  Therefore $(V\cap E_n^2)/V^2$ is a non-trivial
  $\Gal(k_n/k)$-module of order $2$-power.
  Hence there exists $v\bmod V^2\in(V\cap E_n^2)/V^2$
  which is non-trivial and $\gamma$-invariant.
  This means $v^{1-\gamma}\in V^2$ and hence $v^{1+\gamma}\in V^2$.
  Also we obtain
  $\sqrt v^{1-\gamma}=\pm\sqrt{v^{1-\gamma}}\in V$,
  which shows
  $\braket{V,\sqrt v}^\gamma=\braket{V,\sqrt v}$.
\end{proof}

Also, each $V_n^{(j)}$ is automatically $\tau$-invariant.
\begin{lemma}
  \label{lem:key_tau}
  Let $V$ be a subgroup of $E_n$ containing $c_n$
  and satisfying $V^\gamma=V^\tau=V$.
  Then for any $v\in V\cap E_n^2$, we have $v^{1+\tau}\in V^2$,
  and also $\braket{V,\sqrt v}^\tau=\braket{V,\sqrt v}$.
\end{lemma}
\begin{proof}
  Put $u=\sqrt v^{1+\tau}\in E_{\B_n}$.
  The assumptions $c_n\in V$ and $V^\gamma=V$
  yield that $V$ contains
  \[
    C=\braket{-1}\times C_1\times\dots\times C_n,
  \]
  the cyclotomic unit group of $\B_n$.
  From \cite[Theorem~8.2]{washington_gtm_1997},
  we know $\gind{E_{\B_n}}{C}=h_{\B_n}$.
  Hence we have $u^{h_{\B_n}}\in C\subseteq V$.
  Also $u^2=v^{1+\tau}\in V$
  since $V^\tau=V$.
  It is well-known that $h_{\B_n}$ is odd,
  hence $u\in V$, which means
  $v^{1+\tau}\in V^2$.
  Also we have
  $\braket{V,\sqrt v}^\tau=\braket{V,\sqrt v}$
  similarly as the previous lemma.
\end{proof}

Now we formulate the process of extending $V_n$ as follows.
\begin{definition}
  \label{def:iter}
  For $n\ge0$, we define
  \begin{itemize}
    \item $v_n^{(j)}$ : an element of $E''_n$\quad($1\le j\le a_n-a_{n-1}$),
    \item $V_n^{(j)}$ : a subgroup of $E''_n$\quad($0\le j\le a_n-a_{n-1}$), and
    \item $B_n^{(j)}$ : a basis of $V_n^{(j)}$\quad($0\le j\le a_n-a_{n-1}$)
  \end{itemize}
  as follows (for convenience we let $a_{-1}=0$).

  Put
  $V_0^{(0)}=E_0=\braket{-1,\varepsilon}$ and
  $B_0^{(0)}=(-1,\varepsilon)$,
  where $\varepsilon>1$ is the fundamental unit of $k$.
  For $n\ge1$, iteratively on $n$:
  \begin{enumerate}
    \item Put
    \(
      V_n^{(0)}=V_n=E''_{n-1}\times C_n\times H_n\subseteq E_n
    \).
    \item Let $B_n^{(0)}$ be the concatenation of $B_{n-1}^{(a_{n-1}-a_{n-2})}$ and
    \[
      (c_n,c_n^\gamma,c_n^{\gamma^2},\dots,c_n^{\gamma^{2^{n-1}-1}},\eta_n,\eta_n^\gamma,\eta_n^{\gamma^2},\dots,\eta_n^{\gamma^{2^{n-1}-1}}).
    \]
    \item While $V_n^{(j-1)}\subsetneq E''_n$, iteratively on $j$:
    \begin{enumerate}
      \item Let $(b_0,b_1,\dots,b_{2^{n+1}-1})=B_n^{(j-1)}$.
      \item Define $v_n^{(j)}$ as $v\in V_n^{(j-1)}$ satisfying
      \label{it:vnj}
      \begin{enumerate}
        \item $v=b_{i_1}b_{i_2}\dotsm b_{i_t}$ with $0\le i_1<\dots<i_t\le2^{n+1}-1$ and $t\ge1$,
        \label{it:vnj_prod}
        \item $v\in(E''_n)^2$,
        \label{it:vnj_sq}
        \item $v^{1+\gamma}\in(V_n^{(j-1)})^2$,
        \label{it:vnj_gam}
        \item $i_t$ is minimum among those satisfying the above conditions.
        \label{it:vnj_min}
      \end{enumerate}
      \item Put $V_n^{(j)}=\Braket{V_n^{(j-1)},\sqrt{v_n^{(j)}}}$.
      \item Put $B_n^{(j)}=\Bigl(b_0,\dots,b_{i_t-1},\sqrt{v_n^{(j)}},b_{i_t+1},\dots,b_{2^{n+1}-1}\Bigr)$.
    \end{enumerate}
  \end{enumerate}
\end{definition}

Remark that
on (\labelcref{it:vnj}), such $v$ exists from \cref{lem:key_gam},
and it is unique since
if another sequence $i'_1<\dots<i'_{t'}=i_t$ satisfies those conditions
then their symmetric difference also satisfies the conditions,
which contradicts the minimality of $i_t$.
Also remark that
from \cref{lem:key_gam,lem:key_tau},
$V_n^{(j)}$ is $\gamma$-invariant and $\tau$-invariant,
and from \cref{cor:sinnott}, we have $V_n^{(j)}=E''_n$ when $j=a_n-a_{n-1}$.

\begin{example}
  Let $m=627=3\cdot11\cdot19$.
  We find $a_0=2$ by standard calculation.

  First we put $B_1^{(0)}=(-1,\varepsilon,c_1,\eta_1)$.
  Only $\varepsilon$, $\eta_1$, and $\varepsilon\eta_1$ satisfy
  conditions (\labelcref{it:vnj_prod})--(\labelcref{it:vnj_gam})
  of \cref{def:iter}
  (the determination method will be given in \cref{sect:sq}).
  Among them, $\varepsilon$ satisfies condition (\labelcref{it:vnj_min}).
  Hence we put $v_1^{(1)}=\varepsilon$ and
  $B_1^{(1)}=\Bigl(-1,\sqrt{v_1^{(1)}},c_1,\eta_1\Bigr)$.
  For this basis, only $\eta_1$ satisfies
  conditions (\labelcref{it:vnj_prod})--(\labelcref{it:vnj_gam}).
  Hence we put $v_1^{(2)}=\eta_1$ and
  $B_1^{(2)}=\Bigl(-1,\sqrt{v_1^{(1)}},c_1,\sqrt{v_1^{(2)}}\Bigr)$.
  Now there is not a unit satisfying
  conditions (\labelcref{it:vnj_prod})--(\labelcref{it:vnj_gam})
  for this basis, hence
  we have $V_1^{(2)}=E''_1$ and $a_1-a_0=2$.

  Next we put
  \[
    B_2^{(0)}=\Bigl(-1,\sqrt{v_1^{(1)}},c_1,\sqrt{v_1^{(2)}},c_2,c_2^\gamma,\eta_2,\eta_2^\gamma\Bigr),
  \]
  and continue the iteration in the similar manner:
  \begin{align}
    v_2^{(1)}&=\sqrt{v_1^{(2)}},\\
    B_2^{(1)}&=\Bigl(-1,\sqrt{v_1^{(1)}},c_1,\sqrt{v_2^{(1)}},c_2,c_2^\gamma,\eta_2,\eta_2^\gamma\Bigr),\\
    v_2^{(2)}&=-\sqrt{v_1^{(1)}}c_1\eta_2^{1+\gamma},\\
    B_2^{(2)}&=\Bigl(-1,\sqrt{v_1^{(1)}},c_1,\sqrt{v_2^{(1)}},c_2,c_2^\gamma,\eta_2,\sqrt{v_2^{(2)}}\Bigr),\\
    v_3^{(1)}&=\sqrt{v_1^{(1)}}c_1\sqrt{v_2^{(1)}}c_2^{1+\gamma}\eta_2,\\
    B_3^{(1)}&=\Bigl(-1,\sqrt{v_1^{(1)}},c_1,\sqrt{v_2^{(1)}},c_2,c_2^\gamma,\sqrt{v_3^{(1)}},\sqrt{v_2^{(2)}},c_3,\dots,c_3^{\gamma^3},\eta_3,\dots,\eta_3^{\gamma^3}\Bigr),\\
    v_3^{(2)}&=c_2^{1+\gamma}\sqrt{v_2^{(2)}}\eta_3^{1+\gamma+\gamma^2+\gamma^3},\\
    B_3^{(2)}&=\Bigl(-1,\sqrt{v_1^{(1)}},c_1,\sqrt{v_2^{(1)}},c_2,c_2^\gamma,\sqrt{v_3^{(1)}},\sqrt{v_2^{(2)}},\\&\qquad c_3,\dots,c_3^{\gamma^3},\eta_3,\eta_3^\gamma,\eta_3^{\gamma^2},\sqrt{v_3^{(2)}}\Bigr).
  \end{align}
  The iteration terminates here and cannot be continued in the fourth layer.
  Therefore we obtain $(a_0,a_1,\dots,a_4)=(2,4,6,8,8)$.
\end{example}

\section{Fundamental properties of $V_n^{(j)}$}
\label{sect:fund}

Since $V_n^{(j)}$ and $v_n^{(j)}$ exhibit somewhat patterned behavior,
we will explain it here.
Even if a unit of $k_n$ is square in $k_{n+1}$,
it is not necessarily square in $k_n$;
hence we define slightly shifted versions of $V_n^{(j)}$, $v_n^{(j)}$ and $E''_n$.
\begin{definition}
  For $n\ge0$, we put
  \[
    \delta_n=\gind{E''_n\cap(E''_{n+1})^2}{(E''_n)^2}-1=\begin{cases}
      0&(E''_n\cap(E''_{n+1})^2=(E''_n)^2),\\
      1&(E''_n\cap(E''_{n+1})^2\supsetneq(E''_n)^2).
    \end{cases}
  \]
  (See \cref{prop:index_gam}.)
\end{definition}

\begin{lemma}
  \label{lem:vn11}
  If $\delta_n=1$, then
  \begin{enumerate}
    \item $v_{n+1}^{(1)}\in E''_n\cap(E''_{n+1})^2\smallsetminus(E''_n)^2$,
    \item $(v_{n+1}^{(1)})^{1+\gamma}\in(E''_n)^2$.
  \end{enumerate}
\end{lemma}
\begin{proof}
  Let $(b_0,b_1,\dots,b_{2^{n+2}-1})=B_{n+1}^{(0)}$.
  There exists
  \[
    v=b_{i_1}b_{i_2}\dotsm b_{i_t}\in E''_n\cap(E''_{n+1})^2\smallsetminus(E''_n)^2
  \]
  with $0\le i_1<\dots<i_t\le2^{n+1}-1$ and $t\ge1$.
  It is easy to see that $\sqrt v\notin E_n$.
  Hence $k_n(\sqrt v)=k_{n+1}$,
  and also $k_n(\sqrt{v^\gamma})=k_{n+1}$.
  From the Kummer theory we have $v^{1+\gamma}\in E_n^2$,
  hence $v^{1+\gamma}\in(E''_n)^2\subseteq(V_{n+1}^{(0)})^2$.
  Now we have shown that $v$ satisfies the conditions (\labelcref{it:vnj_prod})--(\labelcref{it:vnj_gam})
  in \cref{def:iter}.
  Write $v_{n+1}^{(1)}=b_{i'_1}\dotsm b_{i'_{t'}}$
  with $i'_1<\dots<i'_{t'}$
  satisfying the conditions (\labelcref{it:vnj_prod})--(\labelcref{it:vnj_min}).
  From the condition (\labelcref{it:vnj_min}), we have $i'_{t'}\le i_t\le2^{n+1}-1$.
  Hence $v_{n+1}^{(1)}\in E''_n\smallsetminus(E''_n)^2$,
  and $v_{n+1}^{(1)}\in(E''_{n+1})^2$ from the condition (\labelcref{it:vnj_sq}).
  Therefore $v$ equals to $v_{n+1}^{(1)}$ up to $(E''_n)^2$ from \cref{prop:index_gam},
  hence $v=v_{n+1}^{(1)}$.
\end{proof}

\begin{definition}
  Let $n\ge0$.
  \begin{enumerate}
    \item Let $\at{n}=a_n+\delta_n$.
    \item Let
    \[
      \Et{n}=\begin{cases}
        E''_n&(\delta_n=0),\\
        \Braket{E''_n,\sqrt{v_{n+1}^{(1)}}}&(\delta_n=1).
      \end{cases}
    \]
    \item Let $\Vt{n}{j}=V_n^{(j+\delta_{n-1})}$ for $0\le j\le a_n-\at{n-1}$,
    and $\Vt{n}{\at{n}-\at{n-1}}=\Et{n}$ if $\delta_n=1$.
    \item Let $\vt{n}{j}=v_n^{(j+\delta_{n-1})}$ for $1\le j\le a_n-\at{n-1}$,
    and $\vt{n}{\at{n}-\at{n-1}}=v_{n+1}^{(1)}$ if $\delta_n=1$.
  \end{enumerate}
  Here we let $\delta_{-1}=\at{-1}=0$ for convenience.
\end{definition}

For $\Et{n}$,
the unit index analogous to \cref{prop:index_gam}
is always $1$.
\begin{lemma}
  \label{lem:et_index}
  $\Et{n}\cap(\Et{n+1})^2=(\Et{n})^2$ for $n\ge0$.
\end{lemma}
\begin{proof}
  If $\delta_n=0$,
  then $\delta_{n+1}=0$ from \cref{prop:index_gam}~(\labelcref{it:index_dec}),
  hence the claim holds by definitions.

  Assume $\delta_n=1$.
  Let $v\in\Et{n}\cap(\Et{n+1})^2$.
  We can write
  \(
    v=u\sqrt{v_{n+1}^{(1)}}^e
  \)
  with some $u\in E''_n$ and $e\in\{0,1\}$.
  Taking the norm we obtain
  \[
    N_{n+1,n}(v)=u^2(-v_{n+1}^{(1)})^e,
  \]
  since $k_n\Bigl(\sqrt{v_{n+1}^{(1)}}\Bigr)=k_{n+1}$ from \cref{lem:vn11}.
  We have $v_{n+1}^{(1)}>0$ and $N_{n+1,n}(v)>0$ since $v$ is totally positive.
  Therefore $e=0$ must hold,
  which yields $v\in E''_n$.
  Since $\Et{n+1}\subseteq k_{n+2}$,
  we have
  $k_n(\sqrt v)\subseteq k_{n+2}$.
  Hence we must have
  $k_n(\sqrt v)=k_n$ or $k_n(\sqrt v)=k_{n+1}=k_n\Bigl(\sqrt{v_{n+1}^{(1)}}\Bigr)$.
  This yields that $v$ or $vv_{n+1}^{(1)}$ is square in $E''_n$,
  hence $v\in(\Et{n})^2$.
\end{proof}

The next lemma shows that
$\Vt{n}{j}$ and $\vt{n}{j}$ satisfy
similar properties to those of $V_n^{(j)}$ and $v_n^{(j)}$.
\begin{lemma}
  \label{lem:vt}
  For $n\ge1$ and $1\le j\le\at{n}-\at{n-1}$,
  we have
  \begin{enumerate}
    \item $a_n\ge\at{n-1}$,
    \label{it:vt_at}
    \item $\Vt{n}{0}=\Et{n-1}\times C_n\times H_n$,
    \label{it:vt_0}
    \item $\vt{n}{j}\in\Vt{n}{j-1}\cap(\Et{n})^2\smallsetminus(\Vt{n}{j-1})^2$,
    \label{it:vt_sq}
    \item $(\vt{n}{j})^{1+\gamma}\in(\Vt{n}{j-1})^2$.
    \label{it:vt_gam}
  \end{enumerate}
\end{lemma}
\begin{proof}
  If $\delta_{n-1}=0$,
  these hold by definitions.
  Assume $\delta_{n-1}=1$.

  If $a_n<\at{n-1}$, then $a_n=a_{n-1}$,
  hence $E''_n=E''_{n-1}\times C_n\times H_n$.
  From \cref{lem:vn11},
  we have
  $v_n^{(1)}\in(E''_n)^2$.
  Hence
  \[
    v_n^{(1)}\in(E''_{n-1}\times C_n\times H_n)^2\cap E''_{n-1}=(E''_{n-1})^2,
  \]
  which is a contradiction.
  This shows (\labelcref{it:vt_at}).

  We have
  \begin{align}
    \Vt{n}{0}=V_n^{(1)}&=\Braket{V_n^{(0)},\sqrt{v_n^{(1)}}}\\
    &=\Braket{E''_{n-1},\sqrt{v_n^{(1)}},C_n\times H_n}\\
    &=\braket{\Et{n-1},C_n\times H_n}.
  \end{align}
  Here $\Et{n-1}\cap(C_n\times H_n)=\{1\}$,
  since $(\Et{n-1})^2\subseteq E''_{n-1}$
  from \cref{lem:vn11}.
  Hence we have (\labelcref{it:vt_0}).

  If $j\le a_n-\at{n-1}$,
  (\labelcref{it:vt_sq}) and (\labelcref{it:vt_gam}) are straightforward.
  So we assume $\delta_n=1$ and $j=\at{n}-\at{n-1}=a_n-a_{n-1}$.
  Then $\vt{n}{j}=v_{n+1}^{(1)}$ and $\Vt{n}{j-1}=V_n^{(j)}=E''_n$.
  Therefore we have (\labelcref{it:vt_sq}) and (\labelcref{it:vt_gam}) from \cref{lem:vn11}.
\end{proof}

Next, we introduce the following notation
to ease the proofs of the later theorems.
\begin{definition}
  Let $I$ be the ideal $(1+X)$ of the polynomial ring $\F_2[X]$.
  For $n\ge1$, we define a homomorphism
  \[
    \pi_n\colon\Vt{n}{0}\longrightarrow\F_2[X]/I^{2^{n-1}}
  \]
  to satisfy
  \[
    \pi_n(u\eta_n^{g(\gamma)})=(g\bmod2)\bmod I^{2^{n-1}},
    \label{eq:pi}
  \]
  where $u\in\Et{n-1}\times C_n$ and $g\in\Z[X]$ is a polynomial of
  degree less than $2^{n-1}$.
\end{definition}

\begin{lemma}
  \label{lem:pivf}
  In the above definition,
  \cref{eq:pi} holds
  even without the restriction on the degree of $g$.
  Especially,
  for $v\in\Vt{n}{0}$ and $f\in\Z[X]$,
  we have $\pi_n(v^{f(\gamma)})=f(X)\pi_n(v)$.
\end{lemma}
\begin{proof}
  Let $v=u\eta_n^{g(\gamma)}$.
  We can write
  \[
    g(X)=(1+X^{2^{n-1}})h_1(X)+h_0(X)
  \]
  with some $h_0,h_1\in\Z[X]$ and $\deg h_0<2^{n-1}$.
  Then we have
  \[
    v=u\eta_n^{(1+\gamma^{2^{n-1}})h_1(\gamma)+h_0(\gamma)}
    =\pm u\eta_{n-1}^{h_1(\gamma)}\eta_n^{h_0(\gamma)}
  \]
  using \cref{prop:norm_gam}.
  Hence $\pi_n(v)=(h_0\bmod2)\bmod I^{2^{n-1}}$.
  On the other hand,
  $1+X^{2^{n-1}}\equiv(1+X)^{2^{n-1}}\pmod2$
  holds from the binomial theorem.
  Hence $(g-h_0)\bmod2\in I^{2^{n-1}}$,
  which yields $\pi_n(v)=(g\bmod2)\bmod I^{2^{n-1}}$.
\end{proof}

\begin{lemma}
  \label{lem:etc}
  For $n\ge1$, we have
  \[
    (\Et{n-1}\times C_n)\cap(\Et{n})^2=(\Et{n-1})^2\times C_n^2.
  \]
\end{lemma}
\begin{proof}
  Write $v\in(\Et{n-1}\times C_n)\cap(\Et{n})^2$ as
  \(
    v=uc_n^{f(\gamma)}
  \)
  with some $u\in\Et{n-1}$ and $f\in\Z[X]$
  of degree less than $2^{n-1}$.
  Assume $f\not\equiv0\pmod2$.
  Then we can write
  \[
    f(X)=(1+X)^j(1+(1+X)g(X))+2h_0(X)
  \]
  with some $0\le j\le2^{n-1}-1$ and $g,h_0\in\Z[X]$.
  Using
  \begin{align}
    (1+X)^{2^n-1}&\equiv1+X+\dots+X^{2^n-1}\pmod2,\\
    (1+X)^{2^n}&\equiv1+X^{2^n}\pmod2,\\
    (1+X)^{2^{n-1}}&\equiv1+X^{2^{n-1}}\pmod2,
  \end{align}
  we obtain
  \begin{align}
    (1+X)^{2^n-j-1}f(X)&=1+X+\dots+X^{2^n-1}+(1+X^{2^n})g(X)+2h_1(X),\\
    (1+X)^{2^{n-1}}&=1+X^{2^{n-1}}+2h_2(X),
  \end{align}
  with some $h_1,h_2\in\Z[X]$.
  Therefore we have
  \begin{align}
    v^{(1+\gamma)^{2^n-j-1}}
    &=u^{(1+\gamma^{2^{n-1}}+2h_2(\gamma))(1+\gamma)^{2^{n-1}-j-1}}c_n^{1+\gamma+\dots+\gamma^{2^n-1}}c_n^{(1+\gamma^{2^n})g(\gamma)+2h_1(\gamma)}\\
    &=u^{(2+2h_2(\gamma))(1+\gamma)^{2^{n-1}-j-1}}N_{n,0}(c_n)c_n^{2g(\gamma)+2h_1(\gamma)}.
  \end{align}
  Since $v$ is totally positive, the left-hand side is positive.
  However the right-hand side is negative since $N_{n,0}(c_n)=-1$ from \cref{prop:norm_gam},
  which is a contradiction.
  Hence we have $f\equiv0\pmod2$,
  and also $u\in\Et{n-1}\cap(\Et{n})^2=(\Et{n-1})^2$
  from \cref{lem:et_index}.
  This proves the lemma.
\end{proof}

\begin{lemma}
  \label{lem:piv_0}
  Let $v\in\Vt{n}{0}\cap(\Et{n})^2$.
  If $\pi_n(v)=0$, then $\sqrt v\in\Vt{n}{0}$.
\end{lemma}
\begin{proof}
  We write $v=u\eta_n^{g(\gamma)}$ with
  $u\in\Et{n-1}\times C_n$ and $g\equiv0\pmod2$.
  Then from \cref{lem:etc},
  we have $u\in(\Et{n-1})^2\times C_n^2$.
  Hence $v$ is square in
  $\Et{n-1}\times C_n\times H_n=\Vt{n}{0}$
  from \cref{lem:vt}.
\end{proof}

Now we describe a pattern
for the first $2^{n-1}$ iterations
in the $n$-th layer.
\begin{theorem}
  \label{thm:fund_1}
  Let $n\ge1$ and $0\le j\le\min\{\at{n}-\at{n-1},2^{n-1}\}$.
  Then we have
  \begin{enumerate}
    \item $\vt{n}{j}\in\Vt{n}{0}$ if $j\ge1$,
    \label{it:fund_1_0}
    \item $(\Vt{n}{j})^2\subseteq\Vt{n}{0}$,
    \label{it:fund_1_sq}
    \item $\sqrt v\in\Vt{n}{j}\iff\pi_n(v)\in I^{2^{n-1}-j}/I^{2^{n-1}}$
    for $v\in\Vt{n}{0}\cap(\Et{n})^2$.
    \label{it:fund_1_pi}
  \end{enumerate}
\end{theorem}
\begin{proof}
  Induction on $j$.
  If $j=0$, \cref{lem:piv_0} yields (\labelcref{it:fund_1_pi}).

  Let $j\ge1$
  and assume (\labelcref{it:fund_1_0})--(\labelcref{it:fund_1_pi})
  hold for $j-1$.
  Since $(\vt{n}{j})^2,(\vt{n}{j})^{1+\gamma}\in(\Vt{n}{j-1})^2$,
  we have $(\vt{n}{j})^2,(\vt{n}{j})^{1+\gamma}\in\Vt{n}{0}\cap(\Et{n})^2$
  from the induction hypothesis.
  Write
  \(
    (\vt{n}{j})^2=u\eta_n^{g(\gamma)}
  \)
  with some $u\in\Et{n-1}\times C_n$
  and $g\in\Z[X]$ with $\deg g<2^{n-1}$.
  Then $\pi_n((\vt{n}{j})^2)=(g\bmod2)\bmod I^{2^{n-1}}$.
  From \cref{lem:pivf}
  we have
  \[
    (1+X)\pi_n((\vt{n}{j})^2)=\pi_n((\vt{n}{j})^{2+2\gamma})
    =2\pi_n((\vt{n}{j})^{1+\gamma})=0.
  \]
  Hence $g\bmod2\in I^{2^{n-1}-1}$.
  Since $\deg g<2^{n-1}$,
  we can write
  \[
    g(X)=e(1+X)^{2^{n-1}-1}+2h_0(X)
  \]
  with some $e\in\{0,1\}$ and $h_0\in\Z[X]$.

  If $n=1$, then from our assumption $j=1$ must hold,
  hence (\labelcref{it:fund_1_0}) is shown simply from \cref{lem:vt}.
  If $n\ge2$, it is easy to see that
  \[
    (1+X)^{2^{n-1}}\equiv1+2X^{2^{n-2}}+X^{2^{n-1}}\pmod4
  \]
  holds.
  Using this, we write
  \[
    (1+X)g(X)=e(1+2X^{2^{n-2}}+X^{2^{n-1}})+2(1+X)h_0(X)+4h_1(X)
  \]
  with some $h_1\in\Z[X]$.
  Therefore we have
  \[
    (\vt{n}{j})^{2+2\gamma}
    =\pm u^{1+\gamma}\eta_{n-1}^e\eta_n^{2e\gamma^{2^{n-2}}+2(1+\gamma)h_0(\gamma)+4h_1(\gamma)}.
  \]
  From \cref{lem:etc}, we can write
  $\pm u^{1+\gamma}\eta_{n-1}^e=v^2$
  with $v\in\Et{n-1}\times C_n$.
  Hence we have
  \[
    (\vt{n}{j})^{1+\gamma}
    =\pm v\eta_n^{e\gamma^{2^{n-2}}+(1+\gamma)h_0(\gamma)+2h_1(\gamma)},
  \]
  and from \cref{lem:pivf}
  we have
  \[
    \pi_n((\vt{n}{j})^{1+\gamma})=((eX^{2^{n-2}}+(1+X)h_0(X))\bmod2)\bmod I^{2^{n-1}}.
  \]
  On the other hand,
  from $(\vt{n}{j})^{1+\gamma}\in(\Vt{n}{j-1})^2$
  and the induction hypothesis,
  we have $\pi_n((\vt{n}{j})^{1+\gamma})\in I^{2^{n-1}-j+1}/I^{2^{n-1}}$.
  Hence from the above, $e=0$ must hold.
  Therefore we have $\pi_n((\vt{n}{j})^2)=0$,
  which yields (\labelcref{it:fund_1_0}) from \cref{lem:piv_0}.
  From this and the induction hypothesis,
  (\labelcref{it:fund_1_sq}) can also be shown.
  Also we have
  \[
    (1+X)\pi_n(\vt{n}{j})=\pi_n((\vt{n}{j})^{1+\gamma})\in I^{2^{n-1}-j+1}/I^{2^{n-1}},
  \]
  which yields
  $\pi_n(\vt{n}{j})\in I^{2^{n-1}-j}/I^{2^{n-1}}$.

  Let $v\in\Vt{n}{0}\cap(\Et{n})^2$.
  Suppose $\sqrt v\in\Vt{n}{j}$
  and want to show $\pi_n(v)\in I^{2^{n-1}-j}/I^{2^{n-1}}$.
  If $\sqrt v\in\Vt{n}{j-1}$,
  it holds simply from the induction hypothesis.
  Otherwise, we have $\sqrt{v^{-1}\vt{n}{j}}\in\Vt{n}{j-1}$,
  hence
  \[
    \pi_n(v^{-1}\vt{n}{j})\in I^{2^{n-1}-j+1}/I^{2^{n-1}}
  \]
  from the induction hypothesis.
  Since $\pi_n(\vt{n}{j})\in I^{2^{n-1}-j}/I^{2^{n-1}}$
  as shown above,
  $\pi_n(v)\in I^{2^{n-1}-j}/I^{2^{n-1}}$ follows.
  Conversely,
  suppose $\pi_n(v)\in I^{2^{n-1}-j}/I^{2^{n-1}}$
  and want to show $\sqrt v\in\Vt{n}{j}$.
  If $\pi_n(v)\in I^{2^{n-1}-j+1}/I^{2^{n-1}}$,
  it holds simply from the induction hypothesis.
  Otherwise, we have $\pi_n(v^{-1}\vt{n}{j})\in I^{2^{n-1}-j+1}/I^{2^{n-1}}$
  since $\gind{I^{2^{n-1}-j}}{I^{2^{n-1}-j+1}}=2$.
  Hence we have
  $\sqrt{v^{-1}\vt{n}{j}}\in\Vt{n}{j-1}$
  from the induction hypothesis.
  Therefore we have $\sqrt v\in\Vt{n}{j}$.
  This shows (\labelcref{it:fund_1_pi}).
\end{proof}

\begin{corollary}
  \label{cor:fund_1_eta}
  Let $n\ge1$ and $1\le j\le\min\{\at{n}-\at{n-1},2^{n-1}\}$.
  If we write
  $\vt{n}{j}=u\eta_n^{g(\gamma)}$
  with $u\in\Et{n-1}\times C_n$ and $g\in\Z[X]$ of degree less than $2^{n-1}$,
  then $g$ is the polynomial with $\{0,1\}$-coefficients
  satisfying $g(X)\equiv(1+X)^{2^{n-1}-j}\pmod2$.
\end{corollary}
\begin{proof}
  Induction on $j$.
  Assume that the claim holds for $\vt{n}{j'}$,
  $j'<j$.
  Then, if we let $(b_0,\dots,b_{2^{n+1}-1})=B_n^{(j-1+\delta_{n-1})}$,
  we have
  \[
    b_i=\begin{cases}
      \eta_n^{\gamma^{i-(2^{n+1}-2^{n-1})}}&(2^{n+1}-2^{n-1}\le i\le2^{n+1}-j),\\
      \sqrt{\vt{n}{2^{n+1}-i}}&(2^{n+1}-j<i<2^{n+1}).
    \end{cases}
  \]
  Write $\vt{n}{j}=b_{i_1}\dotsm b_{i_t}$ with $i_1<\dots<i_t$ and
  suppose $i_t>2^{n+1}-j$.
  Taking the minimum $s$ with $i_s>2^{n+1}-j$,
  we obtain $b_{i_s}\notin\Vt{n}{2^{n+1}-i_s-1}$,
  while $b_{i_r}\in\Vt{n}{2^{n+1}-i_s-1}$ for $r\ne s$.
  This contradicts $\vt{n}{j}\in\Vt{n}{0}$.
  Hence we have $i_t\le2^{n+1}-j$,
  which yields
  that $g$ is a polynomial with $\{0,1\}$-coefficients
  and $\deg g\le2^{n-1}-j$.

  From \cref{thm:fund_1}~(\labelcref{it:fund_1_pi}),
  we have $g\bmod2\in I^{2^{n-1}-j}\smallsetminus I^{2^{n-1}-j+1}$.
  Therefore our claim must hold.
\end{proof}

\begin{corollary}
  \label{cor:fund_1_gam}
  \((\Vt{n}{j+1}\smallsetminus\Vt{n}{j})^{1+\gamma}
  \subseteq\Vt{n}{j}\smallsetminus\Vt{n}{j-1}\)
  for $n\ge1$ and $1\le j<\min\{\at{n}-\at{n-1},2^{n-1}\}$.
\end{corollary}
\begin{proof}
  Let $v\in\Vt{n}{j+1}\smallsetminus\Vt{n}{j}$.
  From \cref{thm:fund_1}~(\labelcref{it:fund_1_sq}),
  we have
  $v^2\in\Vt{n}{0}\cap(\Et{n})^2$.
  Hence
  from \cref{thm:fund_1}~(\labelcref{it:fund_1_pi}),
  we have
  \[
    \pi_n(v^2)\in(I^{2^{n-1}-j-1}/I^{2^{n-1}})\smallsetminus(I^{2^{n-1}-j}/I^{2^{n-1}}),
  \]
  which yields
  \[
    \pi_n(v^{2+2\gamma})=(1+X)\pi_n(v^2)\in(I^{2^{n-1}-j}/I^{2^{n-1}})\smallsetminus(I^{2^{n-1}-j+1}/I^{2^{n-1}}).
  \]
  Again
  from \cref{thm:fund_1}~(\labelcref{it:fund_1_pi}),
  we have
  $v^{1+\gamma}\in\Vt{n}{j}\smallsetminus\Vt{n}{j-1}$.
\end{proof}

\begin{corollary}
  \label{cor:fund_1_gam_inv}
  Let $n\ge1$, $1\le j\le\min\{\at{n}-\at{n-1},2^{n-1}\}$,
  and $v\in\Vt{n}{0}\cap(\Et{n})^2$.
  If $\sqrt v^{1+\gamma}\in\Vt{n}{j-1}$,
  then $\sqrt v\in\Vt{n}{j}$.
\end{corollary}
\begin{proof}
  Similarly as \cref{cor:fund_1_gam},
  it is shown
  using \cref{thm:fund_1}~(\labelcref{it:fund_1_pi}).
\end{proof}

The next theorem states
a relationship between
the $n$-th layer and the $(n+1)$-st layer.
\begin{theorem}
  \label{thm:fund_2}
  Let $n\ge1$ and assume $\at{n}-\at{n-1}<2^{n-1}$.
  Then for $1\le j\le\at{n+1}-\at{n}$,
  we have $j\le\at{n}-\at{n-1}$ and
  \[
    (\Vt{n+1}{j}\smallsetminus\Vt{n+1}{j-1})^{1+\gamma^{2^n}}\subseteq\Vt{n}{j}\smallsetminus\Vt{n}{j-1}.
  \]
\end{theorem}
\begin{proof}
  Remark that
  $(\Vt{n+1}{j})^2\subseteq(\Et{n+1})^2\subseteq E''_{n+1}$,
  hence we have
  \[
    (\Vt{n+1}{j})^{2+2\gamma^{2^n}}\subseteq E''_n\cap(\Et{n+1})^2\subseteq(\Et{n})^2=(\Vt{n}{\at{n}-\at{n-1}})^2\subseteq\Vt{n}{0}
  \]
  from \cref{lem:et_index,thm:fund_1}.

  We shall show the theorem by
  induction on $j$.
  First we show the case $j=1$ (harder than the other cases).
  Let $v\in\Vt{n+1}{1}\smallsetminus\Vt{n+1}{0}$.
  Since $v^2,v^{1+\gamma}\in\Vt{n+1}{0}$,
  we write
  \[
    v^2=u_0c_{n+1}^{f_0(\gamma)}\eta_{n+1}^{g_0(\gamma)},\quad
    v^{1+\gamma}=u_1c_{n+1}^{f_1(\gamma)}\eta_{n+1}^{g_1(\gamma)}
    \label{eq:v2_vg}
  \]
  with some $u_0,u_1\in\Et{n}$ and $f_0,g_0,f_1,g_1\in\Z[X]$ of degrees less than $2^n$.
  From \cref{thm:fund_1} we have
  $\pi_{n+1}(v^2)\in I^{2^n-1}/I^{2^n}$ and
  $\pi_{n+1}(v^2)\ne0$,
  which means $g_0\bmod2\in I^{2^n-1}\smallsetminus I^{2^n}$.
  Also we obtain
  \[
    v^{2+2\gamma^{2^n}}=\pm u_0^2c_n^{f_0(\gamma)}\eta_n^{g_0(\gamma)}\in\Vt{n}{0}
  \]
  using \cref{prop:norm_gam}.
  Since $g_0\bmod2\in I^{2^n-1}\subseteq I^{2^{n-1}}$ and
  $c_n^{f_0(\gamma)}\in\Et{n-1}\times C_n$,
  we obtain
  $\pi_n(v^{2+2\gamma^{2^n}})=\pi_n(u_0^2)$
  from \cref{lem:pivf}.
  Therefore from \cref{thm:fund_1},
  our aim is to show that $\pi_n(u_0^2)\in I^{2^{n-1}-1}/I^{2^{n-1}}$ and
  $\pi_n(u_0^2)\ne0$.

  Since $g_0\bmod2\in I^{2^n-1}\smallsetminus I^{2^n}$ and $\deg g_0<2^n$,
  we write
  \[
    (1+X)g_0(X)=1+X^{2^n}+2h(X)
  \]
  with some $h\in\Z[X]$.
  Hence from \cref{eq:v2_vg}, we have
  \[
    u_1^2c_{n+1}^{2f_1(\gamma)}\eta_{n+1}^{2g_1(\gamma)}
    =u_0^{1+\gamma}c_{n+1}^{(1+\gamma)f_0(\gamma)}\eta_{n+1}^{(1+\gamma)g_0(\gamma)}
    =\pm u_0^{1+\gamma}c_{n+1}^{(1+\gamma)f_0(\gamma)}\eta_n\eta_{n+1}^{2h(\gamma)}
  \]
  in $\Vt{n+1}{0}=\Et{n}\times C_{n+1}\times H_{n+1}$.
  Taking the projection to $\Et{n}$, we obtain
  \[
    u_1^2=\pm u_0^{1+\gamma}w\eta_n^{1+2h'(\gamma)},
    \label{eq:u12}
  \]
  where
  $w\in\Et{n-1}\times C_n$ is the projection of $c_{n+1}^{(1+\gamma)f_0(\gamma)}$,
  and $h'\in\Z[X]$.
  Since $u_1^2\in(\Et{n})^2\subseteq\Vt{n}{0}$ as above,
  we have $u_0^{1+\gamma}\in\Vt{n}{0}$.
  Hence from \cref{lem:pivf},
  \[
    (1+X)\pi_n(u_0^2)=\pi_n(u_0^{2+2\gamma})=2\pi_n(u_0^{1+\gamma})=0,
  \]
  which yields $\pi_n(u_0^2)\in I^{2^{n-1}-1}/I^{2^{n-1}}$.

  Suppose $\pi_n(u_0^2)=0$.
  Then we have
  $u_0\in\Vt{n}{0}$
  from \cref{lem:piv_0}.
  Hence from \cref{eq:u12,lem:pivf},
  we obtain
  $\pi_n(u_1^2)\notin I/I^{2^{n-1}}$.
  Since $\at{n}-\at{n-1}<2^{n-1}$,
  it yields
  $\pi_n(u_1^2)\notin I^{2^{n-1}-(\at{n}-\at{n-1})}/I^{2^{n-1}}$,
  which means
  $u_1\notin\Vt{n}{\at{n}-\at{n-1}}=\Et{n}$
  from \cref{thm:fund_1}, a contradiction.
  Therefore $\pi_n(u_0^2)\ne0$.
  Note that this yields $\Vt{n}{0}\subsetneq\Et{n}$,
  hence we have $1\le\at{n}-\at{n-1}$.

  Let $j\ge2$ and assume that the theorem holds for $j-1$.
  Let $v\in\Vt{n+1}{j}\smallsetminus\Vt{n+1}{j-1}$.
  Then $v^{1+\gamma}\in\Vt{n+1}{j-1}\smallsetminus\Vt{n+1}{j-2}$
  from \cref{cor:fund_1_gam}.
  Hence we have
  \[
    (v^{1+\gamma^{2^n}})^{1+\gamma}=(v^{1+\gamma})^{1+\gamma^{2^n}}\in\Vt{n}{j-1}\smallsetminus\Vt{n}{j-2}
  \]
  from the induction hypothesis.
  Here $(v^{1+\gamma^{2^n}})^{1+\gamma}\notin\Vt{n}{j-2}$ yields
  $v^{1+\gamma^{2^n}}\notin\Vt{n}{j-1}$,
  which shows $\Vt{n}{j-1}\subsetneq\Et{n}$.
  This means $j\le\at{n}-\at{n-1}$.
  Also we obtain $v^{1+\gamma^{2^n}}\in\Vt{n}{j}\smallsetminus\Vt{n}{j-1}$
  from \cref{cor:fund_1_gam_inv}.
\end{proof}

From the above,
we obtain the following generalization of
\cite[Corollary~4.4]{fukudakomatsu_funct_2014}.
\begin{corollary}
  \label{cor:fkot}
  Let $n\ge1$ and assume
  $a_n-a_{n-1}<2^{n-1}$.
  Then $a_i-a_{i-1}\ge a_{i+1}-a_i$ for all $i\ge n$.
\end{corollary}
\begin{proof}
  It suffices to show $a_n-a_{n-1}\ge a_{n+1}-a_n$.
  The sequence $(\delta_n)_n$ is non-increasing
  from \cref{prop:index_gam}~(\labelcref{it:index_dec}).
  Hence we have
  \[
    \at{n}-\at{n-1}\le a_n-a_{n-1}<2^{n-1}.
  \]
  From \cref{thm:fund_2} we obtain
  \(
    \at{n+1}-\at{n}\le\at{n}-\at{n-1}
  \), or equivalently,
  \[
    a_{n+1}-a_n\le a_n-a_{n-1}-\delta_{n+1}+2\delta_n-\delta_{n-1}.
  \]
  If $\delta_{n+1}=\delta_n$, this yields $a_n-a_{n-1}\ge a_{n+1}-a_n$.
  So we assume $\delta_{n-1}=\delta_n=1$ and $\delta_{n+1}=0$,
  and also assume that the above inequality is an equality.
  Note that we have $a_n>a_{n-1}$ from Fukuda's theorem~\cite{fukuda_japan_1994}.
  Using \cref{thm:fund_2} for $j=\at{n+1}-\at{n}=a_{n+1}-a_n-1=a_n-a_{n-1}$,
  we obtain
  \begin{align}
    N_{n+1,n}(E''_{n+1}\smallsetminus V_{n+1}^{(a_{n+1}-a_n-1)})
    &=(\Vt{n+1}{a_{n+1}-a_n-1}\smallsetminus\Vt{n+1}{a_{n+1}-a_n-2})^{1+\gamma^{2^n}}\\
    &\subseteq\Vt{n}{a_n-a_{n-1}}\smallsetminus\Vt{n}{a_n-a_{n-1}-1}
    =\Et{n}\smallsetminus E''_n,
  \end{align}
  which is a contradiction.
  Hence the above inequality is strict,
  which shows our claim.
\end{proof}

In \cite[Theorem~4.1]{fukudakomatsuozakitsuji_funct_2016},
they showed that
for any $\Z_l$-extension of any number field,
if $a_n-a_{n-1}<l^n-l^{n-1}$ then
the $\lambda$-invariant is at most $a_n-a_{n-1}$.
The above corollary yields a stronger conclusion,
although it applies only for $\Z_2$-extensions of real quadratic fields.

As examples which satisfy the assumption of
the above corollary,
we provide the following remarkable cases.
\begin{corollary}
  \label{cor:p_ne_1}
  Let $n\ge1$.
  If $m=p$ is a prime with $p\not\equiv1\pmod{2^{n+2}}$,
  then $a_n-a_{n-1}<2^{n-1}$
  (and hence $a_i-a_{i-1}\ge a_{i+1}-a_i$ for all $i\ge n$ from \cref{cor:fkot}).
\end{corollary}
\begin{proof}
  We can assume $n\ge2$
  (and $p\equiv1\pmod8$).
  From \cref{prop:norm_tau},
  we have
  \[
    \eta_n^{1+\tau}=c_n^{\gamma^{-1}+\gamma^{-2}+\dots+\gamma^{-r_p}}
    =c_n^{g(\gamma)},
  \]
  where $r_p$ is the integer satisfying
  $0\le r_p<2^{n+1}$ and $p\equiv\pm3^{r_p}\ (2^{n+3})$,
  and
  \[
    g(X)=X^{2^{n+1}-1}+X^{2^{n+1}-2}+\dots+X^{2^{n+1}-r_p}\in\Z[X].
  \]
  We write
  \[
    g(X)=(1+X^{2^{n-1}})h_1(X)+h_0(X)
  \]
  with some $h_0,h_1\in\Z[X]$ and $\deg h_0<2^{n-1}$.

  Assume $a_n-a_{n-1}\ge2^{n-1}$.
  Then, from \cref{cor:fund_1_eta},
  we can write
  \(
    v_n^{(2^{n-1})}=uc_n^{f(\gamma)}\eta_n
  \)
  with some $u\in E''_{n-1}$ and $f\in\Z[X]$ with $\deg f<2^{n-1}$.
  Hence,
  \[
    (v_n^{(2^{n-1})})^{1+\tau}=u^{1+\tau}c_n^{2f(\gamma)+g(\gamma)}
    =u^{1+\tau}(-c_{n-1})^{h_1(\gamma)}c_n^{2f(\gamma)+h_0(\gamma)}
  \]
  is square in $E''_n$.
  Therefore $h_0\equiv0\pmod2$ must hold
  from \cref{lem:etc}.
  This yields $g(X)\equiv(1+X)^{2^{n-1}}h_1(X)\pmod2$,
  and hence
  \[
    X^{2^{n+1}-r_p}(1+X^{r_p})\equiv(1+X)g(X)\equiv(1+X)^{2^{n-1}+1}h_1(X)\pmod2.
  \]
  Hence $(1+X^{r_p})\bmod2$
  belongs to $I^{2^{n-1}+1}$.

  Assume that $r_p$ is not divisible by $2^n$,
  and let $j\le n-1$ be the $2$-adic valuation of $r_p$.
  Since $r_p-2^j$ is divisible by $2^{j+1}$,
  we have $X^{r-2^j}\equiv1\pmod{I^{2^{j+1}}}$.
  Hence we obtain
  \begin{align}
    (1+X)^{2^j}\bmod2&=(1+X^{r_p}+X^{2^j}-X^{r_p})\bmod2\\
    &\in I^{2^{n-1}+1}+I^{2^{j+1}}
    \subseteq I^{2^j+1},
  \end{align}
  which is a contradiction.
  Therefore $r_p$ is divisible by $2^n$,
  which yields
  $3^{r_p}\equiv1\pmod{2^{n+2}}$.
\end{proof}

\section{Determining whether $v$ is square}
\label{sect:sq}

In \cref{sect:iter},
we explained the outline of our iteration,
without how to determine whether
a given unit $v$ is square in $E_n$.
The following theorem is our fundamental theorem
for this purpose.
\begin{theorem}
  \label{thm:main}
  If $m$ is a prime (with $m\equiv1\pmod8$),
  then for $v\in E_n$,
  $v\in E_n^2$ if and only if
  $v$ satisfies all the following conditions:
  \begin{enumerate}
    \item $v^{1+\gamma}\in E_n^2$,
    \label{it:main_1}
    \item $v^{1+\tau}\in E_n^2$,
    \label{it:main_2}
    \item $\sqrt{v^{1+\gamma}}^{1+\tau}=\sqrt{v^{1+\tau}}^{1+\gamma}$,
    \label{it:main_3}
    \item $v>0$.
    \label{it:main_4}
  \end{enumerate}

  Furthermore,
  if $m$ is not necessarily a prime,
  let $p_1,\dots,p_t$ be the prime divisors of $m$.
  Take prime ideals $\mathfrak{q}_1,\dots,\mathfrak{q}_t$ of $k_n$
  so that
  they split in $k_n/\Q$ and
  the Artin symbols
  \[
    \jac{k_{n+1}(\sqrt{p_1},\dots,\sqrt{p_t})/k_n}{\mathfrak{q}_i}
    \quad(i=1,...,t)
  \]
  span the whole group $\Gal(k_{n+1}(\sqrt{p_1},\dots,\sqrt{p_t})/k_n)$.
  Then for $v\in E_n$,
  $v\in E_n^2$ if and only if
  $v$ satisfies the above (\labelcref{it:main_1})--(\labelcref{it:main_4}) and
  additionally the following condition:
  \begin{enumerate}
    \setcounter{enumi}{4}
    \item $v$ is square modulo $\mathfrak{q}_i$ for all $i=1,\dots,t$.
    \label{it:main_5}
  \end{enumerate}
\end{theorem}

Remark that the equality of (\labelcref{it:main_3})
holds if both sides are squared,
hence only their signs matter.

To prove the theorem, we show the following lemma.
\begin{lemma}
  \label{lem:total}
  Let $v$ be an element of $k_n$.
  If both $v^{1+\gamma}$ and $v^{1+\tau}$ are totally positive,
  then $v$ is totally positive or totally negative.
\end{lemma}
\begin{proof}
  Any conjugation of $v$ is written as $v^{\gamma^i\tau^j}$
  with some integers $i,j$.
  Since
  \[
    X^iY^j\equiv1\pmod{(X-1,Y-1)}
  \]
  in $\Z[X,Y]$,
  we can write
  \[
    v^{\gamma^i\tau^j}=v^{1+(\gamma-1)f(\gamma,\tau)+(\tau-1)g(\gamma,\tau)}
  \]
  with some $f,g\in\Z[X,Y]$.
  Since $v^{1-\gamma}$ and $v^{1-\tau}$ are totally positive,
  $v^{(\gamma-1)f(\gamma,\tau)}$ and
  $v^{(\tau-1)g(\gamma,\tau)}$ are positive.
  Hence $v^{\gamma^i\tau^j}$ and $v$ have the same sign.
\end{proof}

\begin{proof}[Proof of \cref{thm:main}]
  If $v=u^2$ with $u\in E_n$, we have
  \[
    \sqrt{v^{1+\gamma}}^{1+\tau}
    =(\pm u^{1+\gamma})^{1+\tau}
    =u^{(1+\gamma)(1+\tau)}
  \]
  and
  \[
    \sqrt{v^{1+\tau}}^{1+\gamma}
    =(\pm u^{1+\tau})^{1+\gamma}
    =u^{(1+\tau)(1+\gamma)},
  \]
  which coincide.
  This shows (\labelcref{it:main_3}),
  and the other conditions are obvious.

  Conversely, assume the conditions (\labelcref{it:main_1})--(\labelcref{it:main_4}).
  Put $L=k_{n+1}(\sqrt v)$.

  First we show that $L$ is Galois over $\Q$.
  Let $\widetilde{\gamma},\widetilde{\tau}\colon L\to\C$ be
  extensions of $\gamma|_{k_{n+1}},\tau|_{k_{n+1}}$, respectively.
  Then we have
  \[
    \sqrt v^{1+\widetilde{\gamma}}=\pm\sqrt{v^{1+\gamma}}\in k_n
  \]
  from (\labelcref{it:main_1}).
  Hence $\sqrt v^{\widetilde{\gamma}}\in L$,
  which yields $L^{\widetilde{\gamma}}\subseteq L$ since $k_{n+1}/\Q$ is Galois.
  Similarly we have $L^{\widetilde{\tau}}\subseteq L$ using (\labelcref{it:main_2}).
  Therefore $L/\Q$ is Galois.

  Next we show that $L/\Q$ is abelian.
  We have
  \[
    \Gal(L/\Q)=\braket{\widetilde{\gamma},\widetilde{\tau},\sigma},
  \]
  where $\sigma$ is the generator of $\Gal(L/k_{n+1})$.
  From (\labelcref{it:main_1}) we obtain
  \[
    \sqrt v^{(1+\widetilde{\gamma})(1+\widetilde{\tau})}
    =\Bigl(\pm\sqrt{v^{1+\gamma}}\Bigr)^{1+\widetilde{\tau}}
    =\sqrt{v^{1+\gamma}}^{1+\tau}.
  \]
  Similarly,
  from (\labelcref{it:main_2}) we obtain
  \[
    \sqrt v^{(1+\widetilde{\tau})(1+\widetilde{\gamma})}
    =\Bigl(\pm\sqrt{v^{1+\tau}}\Bigr)^{1+\widetilde{\gamma}}
    =\sqrt{v^{1+\tau}}^{1+\gamma}.
  \]
  From (\labelcref{it:main_3}), these coincide.
  Therefore $\widetilde{\gamma}$ and $\widetilde{\tau}$ commute
  on $\sqrt v$, hence on $L$.
  Since $\braket{\sigma}$ is a normal subgroup of $\Gal(L/\Q)$
  whose order is $1$ or $2$,
  $\sigma$ belongs to the center of $\Gal(L/\Q)$.
  Therefore $L/\Q$ is abelian.

  Next we consider the orders of $\widetilde{\gamma}$ and $\widetilde{\tau}$.
  Since $v^{\gamma^{2^n}}=v$, we have
  \[
    \sqrt v^{1-\widetilde{\gamma}^{2^{n+1}}}
    =\sqrt v^{(1-\widetilde{\gamma}^{2^n})(1+\widetilde{\gamma}^{2^n})}
    =\Bigl(\pm\sqrt{v^{1-\gamma^{2^n}}}\Bigr)^{1+\widetilde{\gamma}^{2^n}}=1.
  \]
  From (\labelcref{it:main_2}) and \cref{prop:index_tau},
  we have $v^{1+\tau}\in E_{\B_n}^2$.
  This yields
  \[
    \sqrt v^{1-\widetilde{\tau}^2}
    =\sqrt v^{(1+\widetilde{\tau})(1-\widetilde{\tau})}
    =\Bigl(\pm\sqrt{v^{1+\tau}}\Bigr)^{1-\widetilde{\tau}}
    =\sqrt{v^{1+\tau}}^{1-\tau}=1.
  \]
  Therefore we obtain
  $\widetilde{\gamma}^{2^{n+1}}=\widetilde{\tau}^2=1$,
  which shows that the orders of $\widetilde{\gamma},\widetilde{\tau}$
  equal to the orders of $\gamma|_{k_{n+1}},\tau|_{k_{n+1}}$, respectively.
  Hence we have
  \[
    \Gal(L/\Q)=\braket{\widetilde{\gamma},\widetilde{\tau}}\times\braket{\sigma}.
  \]

  Let $F$ be the field corresponding to $\braket{\widetilde{\gamma},\widetilde{\tau}}$.
  Then $F$ is $\Q$ or a quadratic subfield of $L$
  satisfying $k_{n+1}F=L$ and $k_{n+1}\cap F=\Q$.
  From (\labelcref{it:main_1}), (\labelcref{it:main_2}), (\labelcref{it:main_4}) and \cref{lem:total},
  $v$ is a totally positive unit.
  It yields that $L/\Q$ is unramified outside $2m$.
  Hence $F$ must be contained in $\Q(\sqrt2,\sqrt{p_1},\dots,\sqrt{p_t})$.

  If $m$ is a prime, this yields $F=\Q$ since $\Q(\sqrt2,\sqrt m)=k_1$.
  Hence $\sqrt v\in L=k_{n+1}$, which yields
  $v\in E_n^2$ from \cref{prop:index_gam}~(\labelcref{it:index_mod}).

  Generally, assume that the condition (\labelcref{it:main_5}) holds.
  We have
  \[
    k_n(\sqrt v)\subseteq L=k_{n+1}F\subseteq k_{n+1}(\sqrt{p_1},\dots,\sqrt{p_t}).
  \]
  From our assumption on
  $\mathfrak{q}_1,\dots,\mathfrak{q}_t$,
  if $k_n(\sqrt v)\supsetneq k_n$,
  there exists $i$ such that
  the Artin symbol $\jac{k_n(\sqrt v)/k_n}{\mathfrak{q}_i}$ is non-trivial.
  This contradicts the condition (\labelcref{it:main_5}).
  Therefore we obtain $k_n(\sqrt v)=k_n$,
  which yields $v\in E_n^2$.
\end{proof}

\begin{corollary}
  \label{cor:main}
  The conditions (\labelcref{it:vnj_sq}) and (\labelcref{it:vnj_gam})
  in \cref{def:iter}
  hold if and only if the following conditions hold:
  \begin{enumerate}
    \item $v^{1+\gamma}\in(V_n^{(j-1)})^2$,
    \item $v^{1+\tau}\in(V_n^{(j-1)})^2$,
    \item $\sqrt{v^{1+\gamma}}^{1+\tau}=\sqrt{v^{1+\tau}}^{1+\gamma}$,
    \item $v>0$,
    \item (not necessary when $m$ is a prime) $v$ is square modulo $\mathfrak{q}_i$ for all $i=1,\dots,t$,
    where $\mathfrak{q}_1,\dots,\mathfrak{q}_t$ are those in \cref{thm:main}.
  \end{enumerate}
\end{corollary}
\begin{proof}
  Easily shown from \cref{thm:main,lem:key_tau}.
\end{proof}

These conditions can be rewritten using linear algebra over $\F_2$.
First we introduce a notation.
\begin{definition}
  Let $N=2^{n+1}$
  and let $B=(b_0,\dots,b_{N-1})$
  be a sequence of units in $E_n$.
  For a vector $\mathbf{x}=(x_0,\dots,x_{N-1})\in\Z^N$
  (regarded as a column vector),
  we write
  \[
    B^\mathbf{x}=\prod_{i=0}^{N-1}b_i^{x_i}\in E_n.
  \]
  Similarly, for an $N\times M$-matrix $A=(a_{ij})$ of integer components,
  we write
  \[
    B^A=\left(\prod_{i=0}^{N-1}b_i^{a_{i,0}},\dots,\prod_{i=0}^{N-1}b_i^{a_{i,M-1}}\right).
  \]
\end{definition}

\begin{lemma}
  \label{lem:assoc}
  We have $B^{A_1A_2}=(B^{A_1})^{A_2}$, where $A_1$ is an $N\times N$-matrix and
  $A_2$ is a column vector of size $N$ or an $N\times M$-matrix.
\end{lemma}
\begin{proof}
  Similar as the associativity of matrix multiplications.
\end{proof}

Using a basis $B$,
we can express a unit $v$ as a vector $\mathbf{v}$
satisfying $v=B^\mathbf{v}$.
\begin{lemma}
  \label{lem:lin_sq}
  Let $V$ be a subgroup of $E_n$ containing $E'_n$,
  and let $B=(b_0,b_1,\dots,b_{2^{n+1}-1})$ be a basis of $V$ with $b_0=-1$.
  Then for $\mathbf{v}\in\Z^{2^{n+1}}$,
  $B^\mathbf{v}\in V^2$ if and only if $\mathbf{v}\equiv\mathbf{0}\pmod2$.
\end{lemma}
\begin{proof}
  If $\mathbf{v}=2\mathbf{u}$ with some $\mathbf{u}\in\Z^{2^{n+1}}$,
  then $B^\mathbf{v}=(B^\mathbf{u})^2\in V^2$.
  Conversely, assume $B^\mathbf{v}=(B^\mathbf{u})^2\in V^2$
  with some $\mathbf{u}\in\Z^{2^{n+1}}$.
  Then we have $B^{\mathbf{v}-2\mathbf{u}}=1$.
  This yields that
  the first component of $\mathbf{v}-2\mathbf{u}$
  is even and the other components are $0$,
  since $B$ is a basis.
  Hence all the components of $\mathbf{v}$ are even.
\end{proof}

Using the matrices representing
the maps $v\mapsto v^{1+\gamma}$ and $v\mapsto v^{1+\tau}$,
we can rewrite the conditions of \cref{cor:main}
in terms of linear algebra.
\begin{theorem}
  \label{thm:lin_main}
  Let $V$ and $B$ be those of \cref{lem:lin_sq}.
  Let $G$ be a $2^{n+1}\times2^{n+1}$-matrix of integer components
  satisfying $B^G=B^{1+\gamma}$.
  Similarly, let $T$ be a $2^{n+1}\times2^{n+1}$-matrix of integer components
  satisfying $B^T=B^{1+\tau}$.
  Then all components in the top row of $GT-TG$ are even,
  and all other components of it are zero.
  Further, for $\mathbf{v}\in\Z^{2^{n+1}}$, we have
  \begin{enumerate}
    \item $(B^\mathbf{v})^{1+\gamma}\in V^2\iff G\mathbf{v}\equiv\mathbf{0}\pmod2$,
    \label{it:lin_1}
    \item $(B^\mathbf{v})^{1+\tau}\in V^2\iff T\mathbf{v}\equiv\mathbf{0}\pmod2$,
    \label{it:lin_2}
    \item \(\sqrt{(B^\mathbf{v})^{1+\gamma}}^{1+\tau}=\sqrt{(B^\mathbf{v})^{1+\tau}}^{1+\gamma}\iff
    \frac{1}{2}(GT-TG)\mathbf{v}\equiv\mathbf{0}\pmod2\)
    when $G\mathbf{v}\equiv T\mathbf{v}\equiv\mathbf{0}\pmod2$.
    \label{it:lin_3}
  \end{enumerate}
\end{theorem}
\begin{proof}
  From \cref{lem:assoc}, we have
  \[
    B^{GT}=(B^G)^T=(B^{1+\gamma})^T=(B^T)^{1+\gamma}=B^{(1+\tau)(1+\gamma)}.
  \]
  Similarly we have $B^{TG}=B^{(1+\gamma)(1+\tau)}$,
  which equals to $B^{GT}$.
  Hence all the components of $B^{GT-TG}$ are $1$,
  which yields the first half of our claim.

  Since $(B^\mathbf{v})^{1+\gamma}=(B^{1+\gamma})^\mathbf{v}=(B^G)^\mathbf{v}=B^{G\mathbf{v}}$
  from \cref{lem:assoc},
  we obtain
  (\labelcref{it:lin_1}) from \cref{lem:lin_sq}.
  Also we obtain
  (\labelcref{it:lin_2}) similarly.
  Assume $G\mathbf{v}\equiv T\mathbf{v}\equiv\mathbf{0}\pmod2$.
  Then we have $B^{\frac{1}{2}G\mathbf{v}}=\pm\sqrt{(B^\mathbf{v})^{1+\gamma}}$,
  hence
  \[
    \sqrt{(B^\mathbf{v})^{1+\gamma}}^{1+\tau}
    =(\pm B^{\frac{1}{2}G\mathbf{v}})^{1+\tau}
    =(B^T)^{\frac{1}{2}G\mathbf{v}}
    =B^{\frac{1}{2}TG\mathbf{v}}
  \]
  from \cref{lem:assoc}.
  Similarly we have $\sqrt{(B^\mathbf{v})^{1+\tau}}^{1+\gamma}=B^{\frac{1}{2}GT\mathbf{v}}$.
  Combining them we obtain
  \[
    \frac{\sqrt{(B^\mathbf{v})^{1+\tau}}^{1+\gamma}}{\sqrt{(B^\mathbf{v})^{1+\gamma}}^{1+\tau}}
    =B^{\frac{1}{2}(GT-TG)\mathbf{v}},
  \]
  and the left-hand side is $\pm1$.
  Hence we obtain (\labelcref{it:lin_3}) from \cref{lem:lin_sq}.
\end{proof}

In \cite{fukudakomatsukumakawasasaki_jxm_2026},
we devised a method for calculating $(a_0,a_1,\dots,a_n)$ efficiently
for $m<10^6$ and $n\le11$.
Using \cref{cor:main,thm:lin_main}, this method can be made further faster.
\begin{algorithm}[for specific $m$]
  \label{algo:spec}
  Assume that $m$ is given specifically.
  Let $n_0\ge1$.
  This algorithm computes $(a_0,a_1,\dots,a_{n_0})$.
  \begin{enumerate}
    \item Compute $h_k$ and set $a_0$ to its $2$-adic valuation.
    Initialize $n\leftarrow0$.
    \item Take $\mathfrak{q}_1,\dots,\mathfrak{q}_t$ as in \cref{thm:main}.
    \item Compute
    \[
      B=(-1,\varepsilon,c_1,\eta_1,\dots,c_{n_0},c_{n_0}^\gamma,\dots,c_{n_0}^{\gamma^{2^{n_0-1}-1}},\eta_{n_0},\eta_{n_0}^\gamma,\dots,\eta_{n_0}^{\gamma^{2^{n_0-1}-1}}).
    \]
    The algorithm for computing $\eta_1,\dots,\eta_{n_0}$ is described in \cite{fukudakomatsukumakawasasaki_jxm_2026}.
    \item Initialize $G$ and $T$ to satisfy $B^G=B^{1+\gamma}$ and $B^T=B^{1+\tau}$.
    The components can be computed using
    \cref{prop:norm_gam,prop:norm_tau,prop:eta0,prop:sign_c}.
    \item Set $n\leftarrow n+1$.
    If $n>n_0$, output $(a_0,\dots,a_{n_0})$ and terminate the algorithm.
    \label{it:loop0}
    \item Set $a_n\leftarrow a_{n-1}$.
    \item Let $B_n=(b_0,\dots,b_{2^{n+1}-1})$
    be the first $2^{n+1}$ components of $B$.
    Compute
    \[
      s_{ij}=\begin{cases}
        0&(b_j\text{ is square modulo }\mathfrak{q}_i),\\1&(\text{otherwise})
      \end{cases}
    \]
    for all $1\le i\le t,\,0\le j\le2^{n+1}-1$.
    Since $\mathfrak{q}_i$ splits in $k_n/\Q$,
    these computations can be reduced to computations of rational Legendre symbols.
    Make a $t\times2^{n+1}$-matrix $S=(s_{ij})$.
    \label{it:loop1}
    \item Let $A$ be the $(2^{n+2}+1+t)\times2^{n+1}$-matrix
    formed by stacking
    \begin{itemize}
      \item the top-left $2^{n+1}\times2^{n+1}$-submatrix $G_n$ of $G$,
      \item the top-left $2^{n+1}\times2^{n+1}$-submatrix $T_n$ of $T$,
      \item the top row of $(G_nT_n-T_nG_n)/2$, and
      \item $S$
    \end{itemize}
    vertically.
    Find a non-trivial $\mathbf{v}=(x_0,\dots,x_{2^{n+1}-1})\in\{0,1\}^{2^{n+1}}$
    which satisfies $A\mathbf{v}\equiv\mathbf{0}\pmod2$
    and the maximum index $r$ with $x_r=1$ is minimum.
    If there is not such $\mathbf{v}$, go to (\labelcref{it:loop0}).
    \item Set $a_n\leftarrow a_n+1$.
    \item Compute a square root $u$ of $\lvert B_n^\mathbf{v}\rvert$
    using the method described in \cite{fukudakomatsukumakawasasaki_jxm_2026},
    and compute signs of $u^{1+\gamma}$ and $u^{1+\tau}$.
    \item Replace the $r$-th component of $B$ by $u$.
    \item Transform $G$ and $T$ by elementary operations to satisfy $B^G=B^{1+\gamma}$ and $B^T=B^{1+\tau}$
    for the new $B$.
    The $(0,r)$-components are obtained by the signs of $u^{1+\gamma}$ and $u^{1+\tau}$.
    Go to (\labelcref{it:loop1}).
  \end{enumerate}
\end{algorithm}

As a note on implementation,
instead of the signs,
we can use quadratic residues modulo
a fixed prime ideal which splits in $k_n/\Q$
and $-1$ is non-residue,
to avoid the possibilities of computational errors.

The algorithm in \cite{fukudakomatsukumakawasasaki_jxm_2026}
requires a lot of computations of Legendre symbols
to find a square unit,
but the above algorithm is faster
because it uses the matrix $A$,
which can be computed in a very short time.

\section{The cases where $m$ is a prime}
\label{sect:prime}

In this section,
we assume $m=p$ is a prime with $p\equiv1\pmod8$.
Under this assumption,
we shall describe some properties of $(a_n)_n$
obtained via our iteration.
\cref{cor:p_ne_1} is one of the interesting properties of $(a_n)_n$.
It is well-known that $a_0=0$ and $\varepsilon^{1+\tau}=-1$ hold.

First, we summarize some fundamental properties of
the first layer in our iteration.
\begin{proposition}
  \label{prop:eta1_gt}
  We have
  \begin{enumerate}
    \item $\eta_1^{1+\gamma}=\pm1\equiv2^{(p-1)/4}\pmod p$,
    \label{it:eta1_g}
    \item $\eta_1^{1+\tau}=(-1)^{(p-1)/8}$.
    \label{it:eta1_t}
  \end{enumerate}
\end{proposition}
\begin{proof}
  (See also \cite[Lemma~4.1]{fukudakomatsu_funct_2014}.)
  From \cref{prop:norm_gam,prop:eta0},
  we know $\eta_1^{1+\gamma}=\eta_0=\pm1$.
  Since $2$ is square modulo $p$,
  we have $2^{(p-1)/4}\equiv\pm1\pmod p$.
  Also we have
  \begin{align}
    \eta_0&=\zeta_4^{-(p-1)/4}N_{\Q(\zeta_{4p})/k(\zeta_4)}(1-\zeta_{4p})\\
    &=\zeta_4^{-(p-1)/4}N_{\Q(\zeta_{4p})/k(\zeta_4)}(1-\zeta_4\zeta_p^{-(p-1)/4})\\
    &\equiv\zeta_4^{-(p-1)/4}(1-\zeta_4)^{(p-1)/2}=2^{(p-1)/4}\pmod{1-\zeta_p}.
  \end{align}
  Hence $\eta_0\equiv2^{(p-1)/4}\pmod p$,
  which shows (\labelcref{it:eta1_g}).
  \cref{prop:norm_tau} yields (\labelcref{it:eta1_t}).
\end{proof}

Let $k'=\Q(\sqrt{2p})$ and
let $\varepsilon'>1$ be the fundamental unit of $k'$.
\begin{proposition}
  \label{prop:hkp}
  The $2$-adic valuation of $h_{k'}$ equals to $a_1+1$.
\end{proposition}
\begin{proof}
  Let $H_{k'}$ be the Hilbert $2$-class field of $k'$
  and $H_{k_1}$ be that of $k_1$.
  From the class field theory, we know that
  \[
    k'\subseteq k_1\subseteq H_{k'}\subseteq H_{k_1},
  \]
  and $H_{k_1}/k'$ is Galois.
  Also the abelianization of $\Gal(H_{k_1}/k')$
  is isomorphic to $\Gal(H_{k'}/k')$,
  which is cyclic from the genus theory.
  Hence, from the property of groups of prime-power orders,
  $H_{k'}=H_{k_1}$ must hold.
  This yields our claim.
\end{proof}

\begin{proposition}
  \label{prop:eta1_kp}
  We have
  $\eta_1^2=(\varepsilon')^{-h_{k'}/2}$.
  Further, if $2^{(p-1)/4}\equiv(-1)^{(p-1)/8}\pmod p$,
  then $\eta_1\in k'$ and
  $h_{k'}$ is divisible by $4$.
\end{proposition}
\begin{proof}
  From \cref{prop:norm_eta1} we have
  \[
    \eta_1^{1+\gamma\tau}=N_{k_1/k'}(\eta_1)=\pm(\varepsilon')^{-h_{k'}/2}.
  \]
  On the other hand, from \cref{prop:eta1_gt} we have
  \[
    \eta_1^{1-\gamma\tau}=\eta_1^{(1+\gamma)-\gamma(1+\tau)}=\pm1\equiv(-1)^{(p-1)/8}2^{(p-1)/4}\pmod p.
  \]
  Multiplying them we obtain
  $\eta_1^2=(\varepsilon')^{-h_{k'}/2}$.

  Assume $2^{(p-1)/4}\equiv(-1)^{(p-1)/8}\pmod p$,
  then we obtain $\eta_1^{\gamma\tau}=\eta_1$ from the above.
  This means $\eta_1\in k'$.
  Also,
  since $\eta_1^2=(\varepsilon')^{-h_{k'}/2}$ and
  $\varepsilon'$ is the fundamental unit of $k'$,
  $h_{k'}/2$ must be even.
\end{proof}

The following proposition determines
the behavior of our iteration in the first layer.
\begin{proposition}
  \label{prop:1st}
  We have the following.
  \begin{enumerate}
    \item If $2^{(p-1)/4}\not\equiv(-1)^{(p-1)/8}\pmod p$,
    then $a_1=0$ and $N_{k'/\Q}(\varepsilon')=1$.
    \item If $p\equiv9\pmod{16}$ and $2^{(p-1)/4}\equiv-1\pmod p$,
    then $a_1=1$, $N_{k'/\Q}(\varepsilon')=-1$,
    and $v_1^{(1)}=\lvert\varepsilon c_1\eta_1\rvert$.
    \item If $p\equiv1\pmod{16}$ and $2^{(p-1)/4}\equiv1\pmod p$,
    then $a_1\ge1$ and $v_1^{(1)}=\lvert\eta_1\rvert$.
    Moreover,
    if $a_1=1$ then $N_{k'/\Q}(\varepsilon')=1$,
    and if $a_1\ge2$ then
    $v_1^{(j)}=\sqrt{v_1^{(j-1)}}$ for $2\le j<a_1$ and
    \[
      v_1^{(a_1)}=\begin{cases}
        \sqrt{v_1^{(a_1-1)}}&(N_{k'/\Q}(\varepsilon')=1),\\
        \varepsilon c_1\sqrt{v_1^{(a_1-1)}}&(N_{k'/\Q}(\varepsilon')=-1).
      \end{cases}
    \]
  \end{enumerate}
\end{proposition}
\begin{proof}
  Since $B_1^{(0)}=(-1,\varepsilon,c_1,\eta_1)$ in our iteration,
  from \cref{cor:main} we obtain that
  $a_1\ge1$ if and only if
  there exist $x_0,x_1,x_2\in\{0,1\}$ such that
  \[
    v=(-1)^{x_0}\varepsilon^{x_1}c_1^{x_2}\eta_1
  \]
  satisfies the conditions of \cref{cor:main}.
  Both
  \[
    v^{1+\gamma}=\varepsilon^{2x_1}(-1)^{x_2}\eta_1^{1+\gamma},\quad
    v^{1+\tau}=(-1)^{x_1}c_1^{2x_2}\eta_1^{1+\tau}
  \]
  are square in $V_1^{(0)}$
  if and only if $2^{(p-1)/4}\equiv(-1)^{x_2}\pmod p$
  and $(p-1)/8\equiv x_1\pmod2$
  from \cref{prop:eta1_gt}.
  Under these conditions,
  we have
  \begin{align}
    \sqrt{v^{1+\gamma}}^{1+\tau}&=(\varepsilon^{x_1})^{1+\tau}=(-1)^{x_1},\\
    \sqrt{v^{1+\tau}}^{1+\gamma}&=(c_1^{x_2})^{1+\gamma}=(-1)^{x_2}.
  \end{align}
  Therefore, if $2^{(p-1)/4}\not\equiv(-1)^{(p-1)/8}\pmod p$,
  then $a_1=0$ must hold
  (this result is known in \cite{ozakitaya_manu_1997}).
  In this case, $N_{k'/\Q}(\varepsilon')=1$ also holds
  from \cref{prop:eta1_kp}, since $h_{k'}/2$ is odd from \cref{prop:hkp}.
  Also, if $2^{(p-1)/4}\equiv(-1)^{(p-1)/8}\pmod p$,
  then $a_1\ge1$ and
  \[
    % v_1^{(1)}=v=\lvert(\varepsilon c_1)^{x_1}\eta_1\rvert.
    v_1^{(1)}=v=\begin{cases}
      \lvert\eta_1\rvert&(p\equiv1\pmod{16}),\\\lvert\varepsilon c_1\eta_1\rvert&(p\equiv9\pmod{16}).
    \end{cases}
  \]

  Assume $2^{(p-1)/4}\equiv(-1)^{(p-1)/8}\pmod p$.
  From \cref{prop:eta1_kp}, we have
  $\eta_1=\pm(\varepsilon')^{-h_{k'}/4}$,
  which yields
  $(-1)^{(p-1)/8}=\eta_1^{1+\tau}=N_{k'/\Q}(\varepsilon')^{-h_{k'}/4}$.
  Hence if $p\equiv9\pmod{16}$, then $a_1=1$ and $N_{k'/\Q}(\varepsilon')=-1$.
  Also, if $p\equiv1\pmod{16}$ and $a_1=1$, then $N_{k'/\Q}(\varepsilon')=1$.

  Assume $2^{(p-1)/4}\equiv(-1)^{(p-1)/8}=1\pmod p$
  and $a_1\ge2$.
  Since
  $\eta_1=\pm(\varepsilon')^{-h_{k'}/4}$,
  we must have $v_1^{(j)}=(\varepsilon')^{-h_{k'}/2^{j+1}}$
  for $1\le j<a_1$.
  Therefore we have
  \[
    \sqrt{v_1^{(a_1-1)}}^{1+\gamma}=\sqrt{v_1^{(a_1-1)}}^{1+\tau}
    =N_{k'/\Q}(\varepsilon')
  \]
  since $h_{k'}/2^{a_1+1}$ is odd from \cref{prop:hkp}.
  Similarly as above,
  we can show that $\sqrt{v_1^{(a_1-1)}}$ (resp.~$\varepsilon c_1\sqrt{v_1^{(a_1-1)}}$)
  satisfies the conditions of \cref{cor:main}
  if $N_{k'/\Q}(\varepsilon')=1$ (resp.~$N_{k'/\Q}(\varepsilon')=-1$).
\end{proof}

As applications of \cref{thm:main},
the following interesting properties of $a_n$
are derived.
\begin{theorem}
  \label{thm:a1a2_n}
  Assume $p\equiv1\pmod{16}$ and $N_{k'/\Q}(\varepsilon')=-1$.
  Then we have $a_1\ge2$.
  Moreover,
  \begin{enumerate}
    \item $a_2>a_1\iff a_1\ge3$ when $(-1)^{(p-1)/16}\eta_1>0$,
    \item $a_2>a_1\iff a_1=2$ when $(-1)^{(p-1)/16}\eta_1<0$.
  \end{enumerate}
  % if $p\equiv1\pmod{32}$, then
  % \begin{enumerate}
  %   \item $a_1\ge3\iff a_2>a_1$ when $\eta_1>0$,
  %   \item $a_1\ge3\iff a_2=a_1$ when $\eta_1<0$.
  % \end{enumerate}
  % If $p\equiv17\pmod{32}$, then
  % \begin{enumerate}
  %   \item $a_1\ge3\iff a_2=a_1$ when $\eta_1>0$,
  %   \item $a_1\ge3\iff a_2>a_1$ when $\eta_1<0$.
  % \end{enumerate}
\end{theorem}
\begin{proof}
  We have $a_1\ge2$ from \cref{prop:1st}.
  Using \cref{cor:main,cor:fund_1_eta},
  we obtain that
  $a_2>a_1$ if and only if
  there exist $x_0,x_1,\dots,x_5\in\{0,1\}$ such that
  \[
    v=(-1)^{x_0}\varepsilon^{x_1}c_1^{x_2}\sqrt{v_1^{(a_1)}}^{x_3}c_2^{x_4+x_5\gamma}\eta_2^{1+\gamma}
  \]
  satisfies the conditions of \cref{cor:main}.
  From \cref{prop:1st},
  we have
  $v_1^{(a_1)}=\varepsilon c_1(\varepsilon')^{-h_{k'}/2^{a_1+1}}$,
  which yields $\sqrt{v_1^{(a_1)}}^{1+\gamma}=\pm\varepsilon$ and
  $\sqrt{v_1^{(a_1)}}^{1+\tau}=\pm c_1$.
  Hence we obtain
  \begin{align}
    v^{1+\gamma}&=\varepsilon^{2x_1}(-1)^{x_2}(\pm\varepsilon)^{x_3}(-c_1)^{x_5}c_2^{x_4-x_5+(x_4+x_5)\gamma}\eta_1\eta_2^{2\gamma},\\
    v^{1+\tau}&=(-1)^{x_1}c_1^{2x_2}(\pm c_1)^{x_3}c_2^{2x_4+2x_5\gamma}
  \end{align}
  using \cref{prop:norm_gam,ex:norm_tau},
  and they are square in $V_2^{(0)}$
  if and only if $x_1=x_3=x_4=x_5=0$ and $(-1)^{x_2}\eta_1>0$.
  Under these conditions,
  we have
  \begin{align}
    \sqrt{v^{1+\gamma}}^{1+\tau}
    &=\Bigl(\pm\sqrt{\lvert\eta_1\rvert}\eta_2^\gamma\Bigr)^{1+\tau}
    =((\varepsilon')^{-h_{k'}/8}\eta_2^\gamma)^{1+\tau}
    =(-1)^{h_{k'}/8+(p-1)/16},\\
    \sqrt{v^{1+\tau}}^{1+\gamma}
    &=(\pm c_1^{x_2})^{1+\gamma}
    =(-1)^{x_2}=\sgn\eta_1.
  \end{align}
  Hence $a_2>a_1$ if and only if
  $(-1)^{2^{a_1-2}+(p-1)/16}\eta_1>0$.
\end{proof}

\begin{theorem}
  \label{thm:a1a2_p}
  Assume $N_{k'/\Q}(\varepsilon')=1$.
  If $a_1\ge1$, then $a_2>a_1$.
  Moreover, if $a_1\ge2$ and $(-1)^{(p-1)/16}\eta_1<0$,
  then $a_2=a_1+1$
  (hence the $\lambda$-invariant is at most $1$
  from \cref{thm:fund_2}).
\end{theorem}
\begin{proof}
  From \cref{prop:1st} and our assumptions, we have
  $p\equiv1\pmod{16}$ and $(v_1^{(a_1)})^{2^{a_1-1}}=\lvert\eta_1\rvert$.
  Hence we obtain
  \begin{align}
    \sqrt{v_1^{(a_1)}}^{1+\gamma}&=\pm\sqrt{(v_1^{(a_1)})^{1+\gamma}}=\pm1,\\
    \sqrt{v_1^{(a_1)}}^{1+\tau}&=\pm\sqrt{(v_1^{(a_1)})^{1+\tau}}=\pm1.
  \end{align}
  If both are $1$, then $\sqrt{v_1^{(a_1)}}$ is square in $E_1$
  from \cref{thm:main}, which is a contradiction.
  Similarly, if both are $-1$, then $\varepsilon c_1\sqrt{v_1^{(a_1)}}$ is square in $E_1$
  from \cref{thm:main}, which is also a contradiction.
  Therefore they have different signs,
  so we define $t\in\{0,1\}$ to satisfy
  \[
    \sqrt{v_1^{(a_1)}}^{1+\tau}=-\sqrt{v_1^{(a_1)}}^{1+\gamma}=(-1)^t.
  \]
  Using \cref{cor:main,cor:fund_1_eta},
  we obtain that
  $a_2>a_1$ if and only if
  there exist $x_0,x_1,\dots,x_5\in\{0,1\}$ such that
  \[
    v=(-1)^{x_0}\varepsilon^{x_1}c_1^{x_2}\sqrt{v_1^{(a_1)}}^{x_3}c_2^{x_4+x_5\gamma}\eta_2^{1+\gamma}
  \]
  satisfies the conditions of \cref{cor:main}.
  Both
  \begin{align}
    v^{1+\gamma}&=\varepsilon^{2x_1}(-1)^{x_2}(-1)^{(t+1)x_3}(-c_1)^{x_5}c_2^{x_4-x_5+(x_4+x_5)\gamma}\eta_1\eta_2^{2\gamma},\\
    v^{1+\tau}&=(-1)^{x_1}c_1^{2x_2}(-1)^{tx_3}c_2^{2x_4+2x_5\gamma}
  \end{align}
  are square in $V_2^{(0)}$ if and only if
  they are positive and $x_4=x_5=0$.
  Under these conditions,
  we have
  \begin{align}
    \sqrt{v^{1+\gamma}}^{1+\tau}
    &=\Bigl(\pm\varepsilon^{x_1}\sqrt{\lvert\eta_1\rvert}\eta_2^\gamma\Bigr)^{1+\tau}
    =(-1)^{x_1+2^{a_1-1}t+(p-1)/16},\\
    \sqrt{v^{1+\tau}}^{1+\gamma}
    &=(\pm c_1^{x_2})^{1+\gamma}
    =(-1)^{x_2}.
  \end{align}
  Therefore,
  $a_2>a_1$ if and only if
  the system of linear equations
  \begin{align}
    x_2+(t+1)x_3&\equiv s\pmod2,\\
    x_1+tx_3&\equiv0\pmod2,\\
    x_1+x_2&\equiv2^{a_1-1}t+\frac{p-1}{16}\pmod2
  \end{align}
  has a solution $(x_1,x_2,x_3)\in\{0,1\}^3$,
  where $\sgn\eta_1=(-1)^s$.
  We can check that this always has a (unique) solution,
  hence $a_2>a_1$ holds.

  In the case where $a_1\ge2$ and $(-1)^{(p-1)/16}\eta_1<0$,
  the solution $(x_1,x_2,x_3)$ of the above equations
  must satisfy $x_3=1$.
  Assume $a_2>a_1+1$.
  Then we can write
  \[
    v_2^{(2)}=\pm\varepsilon^{y_1}c_1^{y_2}\sqrt{v_1^{(a_1)}}^{y_3}c_2^{y_4+y_5\gamma}\eta_2
  \]
  with some $y_1,\dots,y_5\in\{0,1\}$,
  using \cref{cor:fund_1_eta}.
  Hence,
  \[
    (v_2^{(2)})^{1+\gamma}=\pm\varepsilon^{2y_1}c_1^{y_5}c_2^{y_4-y_5+(y_4+y_5)\gamma}\eta_2^{1+\gamma}
  \]
  is square in $E_2$.
  Also
  \(
    v_2^{(1)}=v=\pm\varepsilon^{x_1}c_1^{x_2}\sqrt{v_1^{(a_1)}}\eta_2^{1+\gamma}
  \)
  is square in $E_2$,
  hence so is
  \[
    v/(v_2^{(2)})^{1+\gamma}=\pm\varepsilon^{x_1-2y_1}c_1^{x_2-y_5}\sqrt{v_1^{(a_1)}}c_2^{-y_4+y_5-(y_4+y_5)\gamma}.
  \]
  This contradicts \cref{lem:etc}.
  Therefore $a_2=a_1+1$ holds.
\end{proof}

Although the above proofs are based on our iteration,
unlike \cref{algo:spec},
$p$ is not given specifically.
As in these proofs,
we can narrow down the candidates for $a_n$
for an abstract $p$ to some extent
under the certain conditions.
The following is a recursive algorithm
for executing this computationally.
% In the proof of \cref{thm:a1a2_n},
% we did not need to consider the signs of
% $\sqrt{v_1^{(a_1)}}^{1+\gamma}$ and $\sqrt{v_1^{(a_1)}}^{1+\tau}$,
% but in general,
% we need to distinguish four cases based on such signs.
\begin{algorithm}[for abstract $p$]
  \label{algo:abst}
  Here $p$ is not given specifically.
  Let $n\ge1$ and let $b$ be an integer satisfying $b\equiv1\pmod8$.
  Also $e_{ij}\in\{\pm1\}$ for $1\le i\le n,\,0\le j\le2^{i-1}-1$ are given.
  This algorithm computes candidates for $(a_0,a_1,\dots,a_n)$
  where $p$ runs through primes satisfying $p\equiv b\pmod{2^{n+3}}$ and
  $\sgn\eta_i^{\gamma^j}=e_{ij}$ for all $i,j$.
  \begin{enumerate}
    \item Initialize $a_0\leftarrow0,\,\dots,\,a_n\leftarrow0$.
    \item Initialize $G$ and $T$ using
    \cref{prop:norm_gam,prop:norm_tau,prop:eta0,prop:sign_c},
    and using $e_{ij}$.
    \item Let $A$ be the $(2^{n+2}+1)\times2^{n+1}$-matrix
    formed by stacking $G$, $T$, and the top row of $(GT-TG)/2$ vertically.
    Find a non-trivial $\mathbf{v}=(x_0,\dots,x_{2^{n+1}-1})\in\{0,1\}^{2^{n+1}}$
    which satisfies $A\mathbf{v}\equiv\mathbf{0}\pmod2$
    and the maximum index $r$ with $x_r=1$ is minimum.
    If there is not such $\mathbf{v}$, add $(a_0,\dots,a_n)$ to the list of candidates and
    return to the caller.
    \label{it:rec}
    \item Set $a_i\leftarrow a_i+1$ for all $l\le i\le n$,
    where $l$ is the maximal integer satisfying $2^l\le r$.
    \item For $(s,s')\in\{0,1\}^2$, do:
    \label{it:four}
    \begin{enumerate}
      \item Transform $G$ and $T$ by elementary operations as in \cref{algo:spec}.
      Set the $(0,r)$-component of $G$ to $s$
      and the $(0,r)$-component of $T$ to $s'$.
      \item Recursively call (\labelcref{it:rec}) with these $G$, $T$, and $(a_0,\dots,a_n)$.
    \end{enumerate}
    \item Return to the caller.
  \end{enumerate}
\end{algorithm}

As a note on implementation,
proper use of lazy evaluations can avoid
a brute-force execution of all possibilities of $e_{ij}$.
We obtained the following results
through this algorithm.
\begin{example}
  \label{ex:table}
  Assume $p\equiv17\pmod{32}$.
  All possible sequences $(a_0,a_1,\dots,a_4)$
  which are valid under the conditions given in the captions
  are contained in \cref{tab:1,tab:2,tab:3,tab:4}.
  In particular, \cref{thm:kumakawa} can be derived from the tables.
\end{example}

\begin{table}[p]
  \centering
  \begin{minipage}{0.45\textwidth}
    \centering
    \caption{\\\(N_{k'/\Q}(\varepsilon')=-1,\\\eta_1>0,\,a_2\le5\)}
    \label{tab:1}
    \begin{tabular}{rrrrr}
      % \toprule
      % $(a_0,a_1,a_2,a_3,a_4)$\\
      % \midrule
      % $(0,0,0,0,0)$\\
      % $(0,0,0,0,0)$\\
      % $(0,0,0,0,0)$\\
      % $(0,0,0,0,0)$\\
      % \bottomrule
      \toprule
      $a_0$&$a_1$&$a_2$&$a_3$&$a_4$\\
      \midrule
      $0$&$2$&$3$&$4$&$4$\\
      $0$&$2$&$3$&$4$&$5$\\
      $0$&$2$&$4$&$5$&$5$\\
      $0$&$2$&$4$&$5$&$6$\\
      $0$&$2$&$5$&$6$&$6$\\
      $0$&$2$&$5$&$6$&$7$\\
      $0$&$2$&$5$&$7$&$7$\\
      $0$&$2$&$5$&$7$&$8$\\
      $0$&$2$&$5$&$7$&$9$\\
      $0$&$3$&$3$&$3$&$3$\\
      $0$&$4$&$4$&$4$&$4$\\
      $0$&$5$&$5$&$5$&$5$\\
      \bottomrule
    \end{tabular}
  \end{minipage}
  \begin{minipage}{0.45\textwidth}
    \centering
    \caption{\\\(N_{k'/\Q}(\varepsilon')=-1,\\\eta_1<0,\,a_1\le3\)}
    \label{tab:2}
    \begin{tabular}{rrrrr}
      \toprule
      $a_0$&$a_1$&$a_2$&$a_3$&$a_4$\\
      \midrule
      $0$&$2$&$2$&$2$&$2$\\
      $0$&$3$&$4$&$5$&$6$\\
      $0$&$3$&$5$&$6$&$6$\\
      $0$&$3$&$5$&$7$&$8$\\
      $0$&$3$&$5$&$7$&$9$\\
      $0$&$3$&$6$&$6$&$6$\\
      $0$&$3$&$6$&$7$&$8$\\
      $0$&$3$&$6$&$8$&$9$\\
      $0$&$3$&$6$&$8$&$10$\\
      $0$&$3$&$6$&$9$&$9$\\
      $0$&$3$&$6$&$9$&$10$\\
      $0$&$3$&$6$&$9$&$11$\\
      $0$&$3$&$6$&$9$&$12$\\
      \bottomrule
    \end{tabular}
  \end{minipage}
  \begin{minipage}{0.45\textwidth}
    \centering
    \caption{\\\(N_{k'/\Q}(\varepsilon')=1,\\\eta_1>0,\,a_2\le4\)}
    \label{tab:3}
    \begin{tabular}{rrrrr}
      \toprule
      $a_0$&$a_1$&$a_2$&$a_3$&$a_4$\\
      \midrule
      $0$&$0$&$0$&$0$&$0$\\
      $0$&$1$&$2$&$3$&$4$\\
      $0$&$1$&$3$&$3$&$3$\\
      $0$&$1$&$4$&$4$&$4$\\
      $0$&$1$&$4$&$5$&$5$\\
      $0$&$1$&$4$&$6$&$6$\\
      $0$&$1$&$4$&$6$&$7$\\
      $0$&$1$&$4$&$6$&$8$\\
      $0$&$2$&$3$&$3$&$3$\\
      $0$&$2$&$3$&$4$&$4$\\
      $0$&$2$&$3$&$4$&$5$\\
      $0$&$3$&$4$&$4$&$4$\\
      $0$&$3$&$4$&$5$&$5$\\
      $0$&$3$&$4$&$5$&$6$\\
      \bottomrule
    \end{tabular}
  \end{minipage}
  \begin{minipage}{0.45\textwidth}
    \centering
    \caption{\\\(N_{k'/\Q}(\varepsilon')=1,\\\eta_1<0,\,a_1\le3\)}
    \label{tab:4}
    \begin{tabular}{rrrrr}
      \toprule
      $a_0$&$a_1$&$a_2$&$a_3$&$a_4$\\
      \midrule
      $0$&$0$&$0$&$0$&$0$\\
      $0$&$1$&$2$&$3$&$3$\\
      $0$&$1$&$2$&$3$&$4$\\
      $0$&$1$&$3$&$3$&$3$\\
      $0$&$2$&$4$&$5$&$6$\\
      $0$&$2$&$4$&$6$&$6$\\
      $0$&$2$&$4$&$6$&$7$\\
      $0$&$2$&$4$&$6$&$8$\\
      $0$&$3$&$5$&$6$&$7$\\
      $0$&$3$&$5$&$7$&$7$\\
      $0$&$3$&$5$&$7$&$8$\\
      $0$&$3$&$5$&$7$&$9$\\
      \bottomrule
    \end{tabular}
  \end{minipage}
\end{table}

However, it seems that
there is not necessarily a corresponding $p$
for a candidate $(a_n)_n$ output by the above algorithm.
Although all four combinations of signs are checked
in (\labelcref{it:four}) of \cref{algo:abst},
it seems there are combinations of signs
which do not actually occur.
For example,
the algorithm suggests that
$a_1=2$ may hold under the conditions
$p\equiv1\pmod{16}$, $N_{k'/\Q}(\varepsilon')=-1$ and $\eta_1\eta_2^{1+\gamma}>0$,
but in fact it never holds for $p<10^6$.
We propose the following conjecture.
\begin{conjecture}
  Assume $p\equiv1\pmod{16}$ and $N_{k'/\Q}(\varepsilon')=-1$.
  Then, $h_{k'}$ is divisible by $16$ if and only if
  $\eta_1\eta_2^{1+\gamma}>0$.
  Consequently,
  if $(-1)^{(p-1)/16}\eta_2^{1+\gamma}<0$,
  then $a_1=a_2$ from \cref{thm:a1a2_n}.
\end{conjecture}

% Note that
% by calculating the argument of $\eta_1\eta_2^{1+\gamma}$,
% one can show that
% $\eta_1\eta_2^{1+\gamma}>0$ if and only if
% the cardinality of
% \[
%   \Set{x\in\Z|\frac{3}{16}p<x<\frac{9}{16}p,\,\jac{x}{p}=1}
% \]
% is even.

\subsection*{Acknowledgments}
A prototype for abstract computation method via
\cref{def:iter} was devised by Prof.~Keiichi Komatsu.
I would like to thank Prof.~Komatsu and
the participants of the seminar he organizes.
Also, I would like to thank Dr.~Naoki Kumakawa,
who discovered a special case of \cref{thm:kumakawa}.

\bibliographystyle{plain}
\bibliography{ref}

\begin{thebibliography}{10}

\bibitem{ferrerowashington_ann_1979}
Bruce Ferrero and Lawrence~C. Washington.
\newblock The {Iwasawa} invariant {{\(\mu_p\)}} vanishes for abelian number fields.
\newblock {\em Ann. Math. (2)}, 109:377--395, 1979.

\bibitem{fukuda_japan_1994}
Takashi Fukuda.
\newblock Remarks on {{\(\mathbb{Z}_ p\)}}-extensions of number fields.
\newblock {\em Proc. Japan Acad., Ser. A}, 70(8):264--266, 1994.

\bibitem{fukudakomatsu_funct_2014}
Takashi Fukuda and Keiichi Komatsu.
\newblock {On the Iwasawa $\lambda$-invariant of the cyclotomic $\mathbb{Z}_2$-extension of $\mathbb{Q}(\sqrt{p})$ II}.
\newblock {\em Functiones et Approximatio Commentarii Mathematici}, 51(1):167 -- 179, 2014.

\bibitem{fukudakomatsukumakawasasaki_jxm_2026}
Takashi Fukuda, Keiichi Komatsu, Naoki Kumakawa, and Sosuke Sasaki.
\newblock On the iwasawa $\lambda$-invariant of the cyclotomic $\mathbb{Z}_2$-extension of $\mathbb{Q}(\sqrt{m}\,)$.
\newblock {\em Journal of Experimental Mathematics}, to appear.

\bibitem{fukudakomatsuozakitsuji_funct_2016}
Takashi Fukuda, Keiichi Komatsu, Manabu Ozaki, and Takae Tsuji.
\newblock {On the Iwasawa $\lambda$-invariant of the cyclotomic $\mathbb{Z}_2$-extension of $\mathbb{Q}(\sqrt{p})$, III}.
\newblock {\em Functiones et Approximatio Commentarii Mathematici}, 54(1):7 -- 17, 2016.

\bibitem{greenberg_ajm_1976}
Ralph Greenberg.
\newblock On the {Iwasawa} invariants of totally real number fields.
\newblock {\em Am. J. Math.}, 98:263--284, 1976.

\bibitem{hecke_gtm_1981}
Erich Hecke.
\newblock {\em Lectures on the Theory of Algebraic Numbers. {Transl}. from the {German} by {George} {U}. {Brauer} and {Jay} {R}. {Goldman} with the assistance of {R}. {Kotzen}}, volume~77 of {\em Grad. Texts Math.}
\newblock Springer, Cham, 1981.

\bibitem{iwasawa_abh_1956}
Kenkichi Iwasawa.
\newblock A note on class numbers of algebraic number fields.
\newblock {\em Abh. Math. Semin. Univ. Hamb.}, 20:257--258, 1956.

\bibitem{iwasawa_bull_1959}
Kenkichi Iwasawa.
\newblock On {{\(\varGamma\)}}-extensions of algebraic number fields.
\newblock {\em Bull. Am. Math. Soc.}, 65:183--226, 1959.

\bibitem{iwasawa_ann_1973}
Kenkichi Iwasawa.
\newblock On \(\mathbb{Z}_{\ell}\)-extensions of algebraic number fields.
\newblock {\em Ann. Math. (2)}, 98:246--326, 1973.

\bibitem{ozakitaya_manu_1997}
Manabu Ozaki and Hisao Taya.
\newblock On the {Iwasawa} {{\(\lambda_2\)}}-invariants of certain families of real quadratic fields.
\newblock {\em Manuscr. Math.}, 94(4):437--444, 1997.

\bibitem{sinnott_invent_1980}
W.~Sinnott.
\newblock On the {Stickelberger} ideal and the circular units of an abelian field.
\newblock {\em Invent. Math.}, 62:181--234, 1980.

\bibitem{washington_gtm_1997}
Lawrence~C. Washington.
\newblock {\em Introduction to Cyclotomic Fields.}, volume~83 of {\em Grad. Texts Math.}
\newblock New York, NY: Springer, 2nd ed. edition, 1997.

\end{thebibliography}

\end{document}